\documentclass{amsart}
\usepackage{amsfonts}
\usepackage{bbm}
\usepackage{mathrsfs}
\usepackage{color}
\usepackage{amsmath}
\usepackage{bigdelim}
\usepackage{multirow}
\usepackage{amssymb}
\usepackage{indentfirst}
\usepackage{CJK}
\usepackage{fancyhdr}
\usepackage{amscd,amssymb,amsmath,graphicx,verbatim,xy}
\usepackage[TS1,OT1,T1]{fontenc}

\newtheorem{theorem}{Theorem}[section]
\newtheorem{lemma}[theorem]{Lemma}
\newtheorem{corollary}[theorem]{Corollary}
\newtheorem{proposition}[theorem]{Proposition}

\theoremstyle{definition}
\newtheorem{definition}[theorem]{Definition}
\newtheorem{example}[theorem]{Example}

\theoremstyle{remark}
\newtheorem{remark}[theorem]{Remark}
\numberwithin{equation}{section}

\begin{document}
	
\title[$\boldsymbol{i}$-conjugate for quaternionic matrices and related properties]{$\boldsymbol{i}$-conjugate for quaternionic matrices and related properties}

\author{Cailing Yao}
\address{Cailing Yao, School of Medical Informatics, Changchun University of Chinese Medicine, 130117, Changchun, P. R. China}
\email{yaocl@ccucm.edu.cn}
\author{Bingzhe Hou$^{\ast}$}
\address{Bingzhe Hou, School of Mathematics, Jilin University, 130012, Changchun, P. R. China}
\email{houbz@jlu.edu.cn}
\author{Xiaoqi Feng}
\address{Xiaoqi Feng, School of Mathematics, Jilin University, 130012, Changchun, P. R. China}
\email{fengxiaoqi2011@qq.com}

\subjclass{15B33, 15B57, 15A21.}

\keywords{unitary equivalence, $\boldsymbol{i}$-conjugate, quaternionic matrices,
	 $\boldsymbol{i}$-similarity, polar decomposition}

\begin{abstract}
Motivated by the result that a complex $n\times n$ matrix $A$ 
unitarily equivalent to a real matrix, we extend the conclusion to the quaternion skew field. In this paper, we present a necessary and sufficient condition for a quaternion $n\times n$ matrix $A$ to be unitarily equivalent to a complex matrix. To state the truth more clearly, we put forward the concept which we call $\boldsymbol{i}$-conjugate. Furthermore, we study the concepts related to $\boldsymbol{i}$-conjugate and their properties, such as unitary $\boldsymbol{i}$-congruence, $\boldsymbol{i}$-conjugate normality and $\boldsymbol{i}$-Hermicity in $M_{n}(\mathbb{H})$ as generalizations of the conventional unitary congruence, conjugate normality and Hermicity of matrices in $M_{n}(\mathbb{C})$. Finally, we present a new type of polar decomposition of quaternion matrices.
\end{abstract}
\maketitle

\section{Introduction and preliminaries}
In this paper, we mainly concentrate on the $n\times n$ quaternion matrices. The skew field $\mathbb{H}$ of quaternions was discovered by W. Hamilton in 1843. Each $a\in\mathbb{H}$ has the form $a=a_0+a_1\boldsymbol{i}+a_2\boldsymbol{j}+a_3\boldsymbol{k}$, where $a_0, a_1, a_2, a_3\in\mathbb{R}$ and $\boldsymbol{i}^2=\boldsymbol{j}^2=\boldsymbol{k}^2=\boldsymbol{ijk}=-1$. Denote the real and imaginary parts of $a$  by $\mathrm{Re}(a)=a_0$ and $\mathrm{Im}(a)=a_1\boldsymbol{i}+a_2\boldsymbol{j}+a_3\boldsymbol{k},$ respectively, the modulus of $a$ by
$|a|=\sqrt{a_0^2+a_1^2+a_2^2+a_3^2}$ and the conjugate of $a$ by $\overline{a}=a_0-a_1\boldsymbol{i}-a_2\boldsymbol{j}-a_3\boldsymbol{k}$. An element $a\in\mathbb{H}$ is said to be invertible if there exists a quaternion $b$ such that $ab=ba=1$ and we call $b$ the inverse of $a$.

 Denote by $\mathbb{C}$ the field of complex numbers and $M_{n}(\mathbb{C})$ the set of all $n\times n$ complex matrices, $M_{n}(\mathbb{H})$ the set of all $n\times n$ quaternion matrices. If $Q=(q_{ij}) \in M_{n}(\mathbb{H})$, we use $\overline{Q}=(\overline{q_{ij}}),Q^{T}=(q_{ji})$ and $Q^{*}=\overline{Q^{T}}=(\overline{Q})^{T}$ to denote the conjugate, transpose and conjugate transpose of $Q$, respectively. Obviously  if $Q=\overline{Q}$, then $Q$ is a real matrix. If $Q=Q^{T}$, we say $Q$ is symmetric, if $Q=Q^{*}(Q=-Q^{*})$, then we say that $Q$ is Hermitian(skew Hermitian). For two matrices $A,B\in M_{n}(\mathbb{H})$, if there exists an invertible matrix $P\in  M_{n}(\mathbb{H})$ such that $PAP^{-1}=B$, then we say that $A$ is similar to $B$ and denote by $A\sim B$. especially, when the invertible matrix turns to a unitary matrix $U$, $A$ is said to be unitarily equivalent to $B$ and denote by $A\sim_{u} B$.

Every quaternion  matrix $A$ has the form $A=A_1+\boldsymbol{j}A_2$, where $A_1,\ A_2$
are complex matrices. Denote
$$A_\mathbb{C}=
\begin{pmatrix}
A_1 & -\overline{A_2} \\
A_2 & \overline{A_1}
\end{pmatrix}$$
as the complex representation of $A$.
We can know that $A$ is an invertible quaternion matrix if and only if $A_\mathbb{C}$ is an invertible
complex matrix. Clearly, we have $(A+B)_\mathbb{C}=A_\mathbb{C}+B_\mathbb{C}$, $(AB)_\mathbb{C}=A_\mathbb{C}B_\mathbb{C}$. For a complex matrix
$$M=\begin{pmatrix}
M_1 & -\overline{M_2} \\
M_2 & \overline{M_1} \\
\end{pmatrix}$$
where $M_1,\ M_2$ are complex matrices, we say $M_\mathbb{H}=M_1+\boldsymbol{j}M_2$ is the quaternion representation of $M$ and
$(M_{\mathbb{H}})_{\mathbb{C}}=M$. Then for every quaternion matrix $A$, we have $(A_{\mathbb{C}})_{\mathbb{H}}=A$.

The unitary equivalence problem of complex matrices and quaternion matrices has always been a class of problems that attracts much attention.
W. Specht \cite{WS1940} gave the conditions for any two square complex matrices $A$ and $B$ to be unitarily equivalent by the traditional matrix trace.
As for the case of quaternion matrices, D. Dokovi\'{c} and B. H. Smith \cite{DS2008} constructed six unitary trace invariants for $2\times 2$ quaternion matrices which separate the unitary equivalence classes of such matrices and showed this set is minimal. As far as we know, there are no relevant results in the quaternion case  for the case $n\geq 3$.

Unlike considering unitary equivalence from the perspective of traces, Yu. A$\text{l}^\prime$pin and Kh. Ikramov \cite{AI2003} proved that for any two $n \times n$ complex matrices $A$ and $B$, the following assertions are equivalent:
\begin{enumerate}
\item [(1)]$A$ and $B$ are unitarily equivalent;
\item [(2)]the families $\{ A,A^{*} \}$ and $\{ B,B^{*} \}$ are unitarily equivalent;
\item [(3)]the families $\{ A,A^{*} \}$ and $\{ B,B^{*} \}$ are similar.
\end{enumerate}
Kh. Ikramov also showed the necessary and sufficient conditions for a complex matrix to be unitarily equivalent to a real matrix in \cite{I}, he proved that for $A\in M_{n}(\mathbb{C})$, $A$ is unitarily equivalent to a real matrix if and only if $A$ is unitarily equivalent to $\overline{A}$ by a special type of unitary matrix. However, there is no result about the conditions for a quaternion matrix $A\in M_{n}(\mathbb{H})$ to be unitarily equivalent to a complex matrix. In order to settle the problem, we propose some new definitions.
\begin{definition}\label{11}
Let $A$ be a matrix in $M_n(\mathbb{H})$, define $\overline{A}^{\boldsymbol{i}}=-\boldsymbol{i}A\boldsymbol{i}$,  $A^{*{\boldsymbol{i}}}=-\boldsymbol{i}A^*\boldsymbol{i}$, we say that $\overline{A}^{\boldsymbol{i}}$ is the $\boldsymbol{i}$-conjugate of $A$ and $A^{*{\boldsymbol{i}}}$ is the $\boldsymbol{i}$-transpose of $A$.

\end{definition}
\begin{definition}\label{1022}
Let $A$ and $B$ be  two matrices in $M_n(\mathbb{H})$. We say that $A$ is $\boldsymbol{i}$-similar to $B$,  if there exists an invertible matrix $P$ in $M_n(\mathbb{H})$ such that $PA(\overline{P}^{\boldsymbol{i}})^{-1}=B$ and denote by $A\sim_{\boldsymbol{i}}B$. In particular, we say $A$ is unitarily $\boldsymbol{i}$-congruent to $B$ if there exists a unitary matrix $U$ such that $UA(\overline{U}^{\boldsymbol{i}})^{-1}=UA\overline{U^{*}}^{\boldsymbol{i}}=UAU^{*\boldsymbol{i}}=B$ and denote by $A\sim_{\boldsymbol{i}-u} B$.
\end{definition}
From Definition \ref{1022} it is easy to see that the  $\boldsymbol{i}$-similarity and unitary $\boldsymbol{i}$-congruence are both equivalent relationships, and moreover if the invertible matrix $P$ is complex, $\boldsymbol{i}$-similarity is exactly similarity and if the unitary matrix $U$ is complex, the unitary $\boldsymbol{i}$-congruence is just unitary equivalence.

In further, we extend our research. It is worth to mention that L. Huang \cite{H} promoted the definition of consimilarity of matrices from $M_n(\mathbb{C})$ to $M_n(\mathbb{H})$. She defined the $\boldsymbol{j}$-conjugate of a matrix $A\in M_n(\mathbb{H})$ to be $\widetilde{A}=-\boldsymbol{j}A\boldsymbol{j}$ and the $\boldsymbol{j}$-transpose of a matrix $A\in M_n(\mathbb{H})$ to be $A^{J}=-\boldsymbol{j}A^{*}\boldsymbol{j}$. Let $A,B \in M_n(\mathbb{H})$, if there exists an invertible
matrix $P\in M_n(\mathbb{H})$ such that $\widetilde{P^{-1}}AP=B$, then $A$ is said to be consimilar to $B$. If $P$ is a complex matrix, then the quaternion consimilarity coincides with  the consimilarity in $M_n(\mathbb{C})$. In fact, if we replace $\boldsymbol{j}$ with $\boldsymbol{k}$ in the definitions of $\boldsymbol{j}$-conjugate, $\boldsymbol{j}$-transpose and quaternion consimilarity in \cite{H}, then the corresponding definitions become $\boldsymbol{k}$-conjugate, $\boldsymbol{k}$-transpose, the resulting effect will be exactly the same.

In Section $2$, we bring into light the similarities and differences between the classical similarity and $\boldsymbol{i}$-similarity and show the sufficient and necessary conditions for a matrix $A\in M_n(\mathbb{H})$ to be unitarily equivalent to a complex matrix. We prove that if $A\in M_n(\mathbb{H})$, then $A$ is unitarily equivalent to a complex matrix if and only if $A$ is unitarily equivalent to the $\boldsymbol{i}$-conjugate $\overline{A}^{\boldsymbol{i}}$ via an $\boldsymbol{i}$-Hermitian unitary matrix $P$.

Let $A \in M_{n}(\mathbb{H})$, if  a quaternion $\lambda$ satisfies $A\boldsymbol{x}=\boldsymbol{x}\lambda$ for a nonzero vector $\boldsymbol{x}\in \mathbb{H}^{n}$, then we say $\lambda$ is a right eigenvalue of $A$ and $\boldsymbol{x}$ is the right eigenvector with respect to $\lambda$. Similarly, if  a quaternion $\lambda$ satisfies $A\boldsymbol{x}=\lambda\boldsymbol{x}$ for a nonzero vector $\boldsymbol{x}\in \mathbb{H}^{n}$, then we say $\lambda$ is a left eigenvalue of $A$ and $\boldsymbol{x}$ is the left eigenvector with respect to $\lambda$. The right eigenvalues are similar invariants while the left eigenvalues are not.

In Section $3$, we give the definitions of left and right $\boldsymbol{i}$-eigenvalues and point out the connections between them and the left and right eigenvalues. We also calculate the Jordan canonical form of $A \in M_{n}(\mathbb{H})$ under i-similarity.

Recall that for two complex matrices $A,B\in M_n(\mathbb{C})$, if there exists a unitary complex matrix $U$ such that $UAU^{T}=B$, then we say $A$ is unitarily congruent to $B$, it can be seen that the unitary congruence is an equivalent relationship. If the unitary matrix $U$ belongs to $M_n(\mathbb{R})$, then unitary congruence is just unitary equivalence because $UAU^{T}=UA\overline{U^{T}}=UAU^{*}=B$ for $U$ being real. M. Vuji\v{c}i\'{c}, F. Herbut and G. Vuji\v{c}i\'{c} \cite{VHV1972}, F. Herbut, P. Loncke and M. Vuji\v{c}i\'{c} \cite{HLM1974} solved the problem of the canonical form under unitary congruence for a class of complex matrices $A$ which are called conjugate-normal. It was shown in \cite{HR} that a complex matrix $A$ can be unitarily congruent to an upper triangular matrix if and only if all the eigenvalues of $A\overline{A}$ are real and nonnegative, and the Takagi's factorization was proved by using this train of thought.

Furthermore, if we consider unitary congruence in $M_n(\mathbb{H})$ by the same way, we may find that it is not an equivalent relationship anymore because $(AB)^{T}$ is not equal to $B^{T}A^{T}$ for $A,B\in M_n(\mathbb{H})$ in general. Fortunately, the unitary $\boldsymbol{i}$-congruence given in Definition \ref{1022} can be viewed as the generalization of the common unitary congruence in $M_n(\mathbb{C})$. If the unitary matrix $U$ in the definition of unitary $\boldsymbol{i}$-congruence is complex, then unitary $\boldsymbol{i}$-congruence is just unitary equivalence.

In Section $4$, we prove that every quaternion matrix $A \in M_{n}(\mathbb{H})$ is unitarily  $\boldsymbol{i}$-congruent to an upper triangular matrix. Specially, we also prove when a quaternion matrix is unitarily $\boldsymbol{i}$-congruent to an upper triangular matrix $\Delta$ where the main diagonal elements of $\Delta$ are all nonnegative real numbers. Moreover, we give the conditions for quaternion matrices being unitarily  $\boldsymbol{i}$-congruent to complex matrices.

In addition, the normality of a square matrix $A \in M_{n}(\mathbb{H})$ is preserved by unitary equivalence and the conjugate normality of a complex matrix $A \in M_{n}(\mathbb{C})$ is preserved under unitary congruence, however, both the normality and conjugate normality are not maintained under $\boldsymbol{i}$-congruence in $M_{n}(\mathbb{H})$, we define a concept called $\boldsymbol{i}$-conjugate normal for $A \in M_{n}(\mathbb{H})$ which is analogous to conjugate normality in $M_{n}(\mathbb{C})$. We give the canonical form of an $\boldsymbol{i}$-conjugate normal matrix $A$ under unitary $\boldsymbol{i}$-congruence and show a series of other properties of $\boldsymbol{i}$-conjugate normal matrices. Finally, we present a special polar decomposition by using the $\boldsymbol{i}$-Hermitian matrix and unitary matrix and give the sufficient and necessary conditions for a quaternion matrix to be $\boldsymbol{i}$-conjugate normal through this kind of polar decomposition.

\section{i-similarity and unitary equivalence}

\subsection{i-similarity}
In this section, we discuss the fundamental properties of $\boldsymbol{i}$-similarity.

\begin{lemma}\label{haoduo}
Let $A,B$ be two matrices in $M_n(\mathbb{H})$. Then the following holds.
\begin{enumerate}

    \item $ \overline{\overline{A}^{\boldsymbol{i}}}^{\boldsymbol{i}}=A=(A^{*{\boldsymbol{i}}})^{*{\boldsymbol{i}}}$,
    $(\overline{A}^{\boldsymbol{i}})^* = \overline{A^*}^{\boldsymbol{i}}=A^{*{\boldsymbol{i}}}$, $(A^{*{\boldsymbol{i}}})^{*}=(A^*)^{*{\boldsymbol{i}}}$.

    \item If $A$ is invertible, then $(\overline{A}^{\boldsymbol{i}})^{-1}=\overline{A^{-1}}^{\boldsymbol{i}}$, $(A^{*{\boldsymbol{i}}})^{-1}=(A^{-1})^{*{\boldsymbol{i}}}$.
    \item $\overline{A}^{\boldsymbol{i}}=A$, $A^{*}=(A)^{*{\boldsymbol{i}}}$ if and only if $A \in M_n(\mathbb{C})$.

    \item $\overline{A\boldsymbol{x}}^{\boldsymbol{i}}=\overline{A}^{\boldsymbol{i}}\overline{\boldsymbol{x}}^{\boldsymbol{i}}$ for every $\boldsymbol{x}\in\mathbb{H}^n$.
    \item $\overline{A+B}^{\boldsymbol{i}}=\overline{A}^{\boldsymbol{i}}+\overline{B}^{\boldsymbol{i}}$, $(A+B)^{*{\boldsymbol{i}}}=A^{*{\boldsymbol{i}}}+B^{*{\boldsymbol{i}}}$ for $A,B\in M_n(\mathbb{H})$.
    \item $\overline{AB}^{\boldsymbol{i}}=\overline{A}^{\boldsymbol{i}}\cdot\overline{B}^{\boldsymbol{i}}$,  $(AB)^{*{\boldsymbol{i}}}= B^{*{\boldsymbol{i}}}A^{*{\boldsymbol{i}}}$.

\item $\overline{A^{*{\boldsymbol{i}}}}^{\boldsymbol{i}}=(\overline{A}^{\boldsymbol{i}})^{*{\boldsymbol{i}}}=A^{*}$,

\item $\overline{A^T}^{\boldsymbol{i}}=(\overline{A}^{\boldsymbol{i}})^T=\overline{A^{*{\boldsymbol{i}}}}=(\overline{A})^{*{\boldsymbol{i}}}$,

\item $(A^{*{\boldsymbol{i}}})^{T}=(A^{T})^{*{\boldsymbol{i}}}=\overline{\overline{A}^{\boldsymbol{i}}}=\overline{\overline{A}}^{\boldsymbol{i}}  $,

\item $rank(A)=rank(\overline{A}^{\boldsymbol{i}})=rank(A^{*{\boldsymbol{i}}})=rank(A^{*})$,

\item $A^{*{\boldsymbol{i}}}=A^T$ if and only if $A\in \mathbb{R}\oplus \mathbb{R}\boldsymbol{j}\oplus\mathbb{R}\boldsymbol{k}$,

\item $J\overline{A_{\mathbb{C}}}^{\boldsymbol{i}}J^{-1}=(\overline{A}^{\boldsymbol{i}})_{\mathbb{C}}$, where $J=\begin{pmatrix}
I_n & 0 \\
0 & -I_n
\end{pmatrix}$, $\overline{A_{\mathbb{C}}}^{\boldsymbol{i}}=(\overline{A}^{\boldsymbol{i}})_{\mathbb{C}}$ if and only if $A\in M_{n}(\mathbb{C})$.
\end{enumerate}
\end{lemma}
Notice that for two complex matrices $A$ and $B$, the relationship $\overline{AB}=\overline{A}\ \overline{B}$ always holds. But this is not true for the case one of $A$ and $B$ becomes quaternion matrix generally. The first equation in $(6)$ of Lemma \ref{haoduo} show that the $\boldsymbol{i}$-conjugate in $M_{n}(\mathbb{H})$ has the similar property with the classic conjugate in $M_{n}(\mathbb{C})$.

\begin{lemma}\label{yingwen}
Let $a=a_{0}+a_{1}\boldsymbol{i}+a_{2}\boldsymbol{j}+a_{3}\boldsymbol{k}$ and $a^{\prime}=\sqrt{|a|^{2}-a_{1}^{2}}+a_{1}\boldsymbol{i}$, then $a\sim_{\boldsymbol{i}}a^{\prime}$.
\end{lemma}

\begin{proof}
If $a_{0}\geq 0,\ a_{2}=a_{3}=0$, it is obvious that $a=a_0+a_{1}\boldsymbol{i}=a^{\prime}$. Otherwise, let
$$
p=a_{3}+a_{2}\boldsymbol{i}-(\sqrt{|a|^{2}-a_{1}^{2}}-a_{0})\boldsymbol{k}\neq 0,
$$
it is easy to verify that $pa(\overline{p}^{\boldsymbol{i}})^{-1}=a^{\prime}$. Furthermore, the quaternion $p$ can be taken of unit length.
\end{proof}
For two nonzero quaternions $a=a_{0}+a_{1}\boldsymbol{i}+a_{2}\boldsymbol{j}+a_{3}\boldsymbol{k}$ and $b=b_{0}+b_{1}\boldsymbol{i}+b_{2}\boldsymbol{j}+b_{3}\boldsymbol{k}$, we know that $a\sim b(a\sim_{u} b)$ iff $|a|=|b|$ and $a_{0}=b_{0}$, there is an analogous conclusion for $\boldsymbol{i}$-similarity between two quaternions.

\begin{proposition}\label{QISR}
Let $a=a_{0}+a_{1}\boldsymbol{i}+a_{2}\boldsymbol{j}+a_{3}\boldsymbol{k}, b=b_{0}+b_{1}\boldsymbol{i}+b_{2}\boldsymbol{j}+b_{3}\boldsymbol{k}$ be two quaternions, then $a \sim_{\boldsymbol{i}} b(a \sim_{\boldsymbol{i}-u} b)$ if and only if $|a|=|b|$ and $a_{1}=b_{1}$.
\end{proposition}

\begin{lemma}\label{RIST}
Let $A, B$ be two matrices in $M_n(\mathbb{H})$. Then

\begin{enumerate}
    \item $A \sim_{u} \overline{A}^{\boldsymbol{i}}$.

    \item $A\sim_{\boldsymbol{i}}B \Leftrightarrow\boldsymbol{i}A\sim\boldsymbol{i}B \Leftrightarrow A\boldsymbol{i}\sim\boldsymbol{i}B \Leftrightarrow \boldsymbol{i}A\sim B\boldsymbol{i}  \Leftrightarrow A\boldsymbol{i}\sim B\boldsymbol{i}$.
    \item  $A\sim_{\boldsymbol{i}-u}B \Leftrightarrow\boldsymbol{i}A\sim_{u}\boldsymbol{i}B \Leftrightarrow A\boldsymbol{i}\sim_{u}\boldsymbol{i}B \Leftrightarrow \boldsymbol{i}A\sim_{u} B\boldsymbol{i}  \Leftrightarrow A\boldsymbol{i}\sim_{u} B\boldsymbol{i}$.
\end{enumerate}
\end{lemma}
\begin{proof}

\begin{enumerate}
        \item Let $U=\mathrm{diag}(-\boldsymbol{i},-\boldsymbol{i},\cdots,-\boldsymbol{i})=-\boldsymbol{i}I_{n}$, then $U$ is unitary and $UAU^{*}= \overline{A}^{\boldsymbol{i}}$.

        \item Following from $A\sim_{\boldsymbol{i}}B$, there exists an invertible matrix $P\in M_n(\mathbb{H})$ such that
    \begin{equation}\label{11271}
    B=(\overline{P}^{\boldsymbol{i}})^{-1}AP=-\boldsymbol{i}P^{-1}\boldsymbol{i}AP,
    \end{equation}
    one can see that (\ref{11271}) is equivalent to
    $P^{-1}\boldsymbol{i}AP=\boldsymbol{i}B$, which means the same fact as $\boldsymbol{i}A\sim\boldsymbol{i}B$. Let  $Q=\mathrm{diag}(\boldsymbol{i},\boldsymbol{i},\cdots,\boldsymbol{i})=\boldsymbol{i}I_{n}$, since $Q^{-1}\boldsymbol{i}AQ=A\boldsymbol{i}$, $Q^{-1}\boldsymbol{i}BQ=B\boldsymbol{i}$, we have $\boldsymbol{i}A\sim A\boldsymbol{i}$ and $\boldsymbol{i}B\sim B\boldsymbol{i}$.
    \item The proof of $(3)$ is analogous to that of $(2)$.
    \end{enumerate}
    \end{proof}
L. Wolf \cite{W} proved that the similarity between two quaternion matrices $A$ and $B$ is equivalent to the similarity between their complex representation matrices $A_{\mathbb{C}}$ and $B_{\mathbb{C}}$.

\begin{theorem}[\cite{W}]\label{RQS}
Let $A, B$ be in $M_n(\mathbb{H})$, then $A$ and $B$ are similar in $M_n(\mathbb{H})$ if and only if $A_{\mathbb{C}}$ and
$B_{\mathbb{C}}$ are similar in $M_n(\mathbb{C})$.
\end{theorem}
F. Zhang \cite{Z} proved that if two quaternion matrices $A$ and $B$ are similar, then they have the same characteristic polynomial.
\begin{theorem}[\cite{Z}]\label{eng}
Let $A, B$ be in $M_n(\mathbb{H})$. If $A$ is similar to $B$, then $F_A(\lambda)=F_B(\lambda)$, where $\lambda\in \mathbb{C}$,
$F_A(\lambda)=|\lambda I_{2n}-A_{\mathbb{C}}|$ is the characteristic polynomial of $A$.
\end{theorem}
By Theorem \ref{RQS} and Theorem \ref{eng}, we can obtain following assertions.
\begin{corollary}\label{11241}
Let $A, B$ be in $M_n(\mathbb{H})$, then the following statements hold true:
\begin{enumerate}
\item [1.]$A\sim_{\boldsymbol{i}}B \Leftrightarrow Q_{\boldsymbol{i}}A_{\mathbb{C}}\sim Q_{\boldsymbol{i}}B_{\mathbb{C}}$ where $Q_{\boldsymbol{i}}=\boldsymbol{i}I_{n}\oplus(-\boldsymbol{i}I_{n})$.
\item [2.]$A\sim_{\boldsymbol{i}}B\Rightarrow F_{\boldsymbol{i}A}(\lambda)=F_{\boldsymbol{i}B}(\lambda)$.
\item [3.]$A\sim_{\boldsymbol{i}}B\Rightarrow p(\boldsymbol{i}A)\sim p(\boldsymbol{i}B)$ for any real coefficient polynomial $p$.
\end{enumerate}

\end{corollary}
Actually, the first conclusion in Corollary \ref{11241} is analogous to Theorem $5$ in \cite{JiangCL} essentially, it is worth to note that the complex representation we use is different from the one in \cite{JiangCL}.

A relatively well known conclusion is that any complex matrix can be similar to a symmetric matrix.
\begin{lemma}[\cite{HR}, Theorem $4.4.9$]\label{11272}
Let $A\in M_{n}({\mathbb{C}})$, then there exists a symmetric matrix $S$ and an invertible matrix $P$ such that $PAP^{-1}=S$ .
\end{lemma}
As for the  $\boldsymbol{i}$-similarity of quaternion matrices, we have the analogous conclusion.
\begin{theorem}
Every quaternion matrix $A\in M_{n}({\mathbb{H}})$ is $\boldsymbol{i}$-similar to a symmetric matrix.
\end{theorem}
\begin{proof}
 As well known, every quaternion matrix in $M_{n}({\mathbb{H}})$ is similar to a complex matrix in $M_{n}({\mathbb{C}})$, one can refer to \cite{Z}. We assume the quaternion matrix $\boldsymbol{i}A$ is similar to the complex matrix $\boldsymbol{i}C$. By Lemma \ref{11272}, we may let $\boldsymbol{i}C$ is similar to the symmetric complex matrix $\boldsymbol{i}T$. Since $\boldsymbol{i}T=(\boldsymbol{i}T)^T=\boldsymbol{i}T^T$, we have $T=T^T$, which implies that $T$ is also symmetric.
Furthermore, the similarity among $\boldsymbol{i}A$, $\boldsymbol{i}C$ and $\boldsymbol{i}T$ is equivalent to the $\boldsymbol{i}$-similarity among $A$, $C$ and $T$, hence we draw the conclusion.
\end{proof}

\subsection{Unitary equivalence}

Kh. Ikramov provided a necessary and sufficient condition for a complex matrix to be unitarily equivalent to a real matrix in \cite{I}.

\begin{theorem}[\cite{I}]\label{Kh1}
 A matrix $A\in M_{n}(\mathbb{C})$ can be made real by a unitary equivalence transformation if and only if A and $\overline{A}$ are unitarily equivalent and the unitary equivalence between them can be realized via a transformation matrix $P$ that is simultaneously unitary and symmetric.
\end{theorem}

 Further, based on the result above, we want to know under what conditions can a quaternion matrix be unitarily equivalent to a complex matrix, it would be even better if we could obtain the necessary and sufficient conditions for this question. Before presenting the related theorem, we first introduce a lemma that will be used in the proof process.

\begin{lemma}\label{hahah}
Let $D\in M_{n}(\mathbb{H})$, then $D\overline{D}^{\boldsymbol{i}}=I_{n}$ if and only if there exists an invertible matrix
 $S\in M_{n}(\mathbb{H})$ such that $D=S(\overline{S}^{\boldsymbol{i}})^{-1}$.
\end{lemma}

\begin{proof}
The sufficiency is obvious.

Necessity. Let $\Upsilon=\{\rho\in\mathbb{H};\rho\overline{\rho}^{\boldsymbol{i}}=1\}$, it is evident the set $\Upsilon$ is not empty. Define $S_{q}=Dq+\overline{q}^{\boldsymbol{i}}I_{n}$, where $q\in\Upsilon$.

By the form of $S_{q}$ we obtain $\overline{S_q}^{\boldsymbol{i}}=\overline{D}^{\boldsymbol{i}}\overline{q}^{\boldsymbol{i}}+qI_{n}$, then $D\cdot \overline{S_q}^{\boldsymbol{i}} = S_{q}$. By Theorem $6$ in \cite{Z07}, we may choose a proper $q_{0}\in\Upsilon$ such that the matrix $D+(\overline{q_0}^{\boldsymbol{i}})^{2}\cdot I_n$ is invertible, then $S_{q_{0}}=[D+(\overline{q_0}^{\boldsymbol{i}})^{2}\cdot I_n]q_{0}$ is also invertible and $D=S_{q_{0}}(\overline{S_{q_0}}^{\boldsymbol{i}})^{-1}$, let $S=S_{q_{0}}$, then $D=S(\overline{S}^{\boldsymbol{i}})^{-1}$ as desired.
\end{proof}
We show a necessary and sufficient condition for a quaternion matrix being unitarily equivalent to a complex matrix in the following theorem, in which the unitary quaternion matrix $P$ plays the role of the symmetric unitary complex matrix in Theorem \ref{Kh1}.

\begin{theorem}\label{QC1}
A matrix $A\in M_{n}(\mathbb{H})$ is unitarily equivalent to a complex matrix if and only if A is unitarily equivalent to its $\boldsymbol{i}$-conjugate $\overline{A}^{\boldsymbol{i}}$  by a special unitary matrix P with $P=UU^{*\boldsymbol{i}}$, where $U\in M_{n}(\mathbb{H})$ is a unitary matrix.
\end{theorem}

\begin{proof}
First, let us prove the necessity.

Assume that the complex matrix $B=T^{\ast}AT$, where $T$ is a unitary matrix. Then $\overline{B}^{\boldsymbol{i}}=\overline{T^*}^{\boldsymbol{i}}\overline{A}^{\boldsymbol{i}}\overline{T}^{\boldsymbol{i}}=B$ and
$$
\overline{A}^{\boldsymbol{i}} = (\overline{T^*}^{\boldsymbol{i}})^{-1}T^{\ast}AT(\overline{T}^{\boldsymbol{i}})^{-1}=P^{\ast}AP,
$$
where $P=T(\overline{T}^{\boldsymbol{i}})^{-1}=TT^{*\boldsymbol{i}}$ is unitary.

Now, let us prove the Sufficiency.

Let $\overline{A}^{\boldsymbol{i}}=P^{\ast}AP$ for a unitary matrix $P$ and $P=UU^{*\boldsymbol{i}}$, where $U$ is unitary. Since every quaternion matrix is similar to a complex matrix, $A$ can be written as $A=QMQ^{-1}$, where $M$ is complex. Hence
$$
\overline{A}^{\boldsymbol{i}}=\overline{Q}^{\boldsymbol{i}} M \overline{Q^{-1}}^{\boldsymbol{i}} = P^{\ast}QMQ^{-1}P,
$$
and
$$
(Q^{-1}P\overline{Q}^{\boldsymbol{i}})M=M(Q^{-1}P\overline{Q}^{\boldsymbol{i}}).
$$
Put $D=Q^{-1}P\overline{Q}^{\boldsymbol{i}}$. Then $DM=MD$.
Observe that
\begin{equation}\label{11273}
D\overline{D}^{\boldsymbol{i}}=(Q^{-1}P \overline{Q}^{\boldsymbol{i}}) (\overline{Q^{-1}}^{\boldsymbol{i}} \overline{P}^{\boldsymbol{i}}Q)=Q^{-1}(P\overline{P}^{\boldsymbol{i}})Q=I_n.
\end{equation}
The last equality in (\ref{11273}) is explained by $P\overline{P}^{\boldsymbol{i}}=UU^{*\boldsymbol{i}}\overline{U}^{\boldsymbol{i}}U^{*}=I_{n}$.

According to Lemma \ref{hahah}, $D$ can be represented as $D=S({\overline{S}^{\boldsymbol{i}}})^{-1}$, thus
$$
S({\overline{S}^{\boldsymbol{i}}})^{-1}M=MS({\overline{S}^{\boldsymbol{i}}})^{-1},
$$
or equivalently
$$
({\overline{S}^{\boldsymbol{i}}})^{-1}M\overline{S}^{\boldsymbol{i}}=S^{-1}MS.
$$
Let $L=S^{-1}MS$, then $L$ is complex.
Now we have
$$
P=QD({\overline{Q}^{\boldsymbol{i}}})^{-1} = QS({\overline{S}^{\boldsymbol{i}}})^{-1}({\overline{Q}^{\boldsymbol{i}}})^{-1}
=V({\overline{V}^{\boldsymbol{i}}})^{-1},
$$
where $V=QS$.
We have $V({\overline{V}^{\boldsymbol{i}}})^{-1}=U({\overline{U}^{\boldsymbol{i}}})^{-1}$,
and
$$
U^{\ast}V=U^{-1}V=({\overline{U}^{\boldsymbol{i}}})^{-1}\overline{V}^{\boldsymbol{i}},
$$
then $R=U^{\ast}V$ is a complex matrix.
Finally,
$$
A=QMQ^{-1}=VS^{-1}MSV^{-1}=VLV^{-1}=U(RLR^{-1})U^{\ast}.
$$
The matrix $RLR^{-1}$ is complex because both $R$ and $L$ are complex. While $U$ is unitary, thus $A$ is unitarily equivalent to a complex matrix.
\end{proof}

\section{i-eigenvalue and i-Jordan canonical form}
In this section, we will discuss the $\boldsymbol{i}$-eigenvalues and $\boldsymbol{i}$-Jordan canonical form of a quaternion matrix $A \in M_n(\mathbb{H})$.
\subsection{i-eigenvalue}

\begin{definition}
Let $A$ be a matrix in $M_n(\mathbb{H})$. If $A\overline{\boldsymbol{x}}^{\boldsymbol{i}}=\omega\boldsymbol{x}$ holds for a quaternion $\omega$ and a nonzero vector $\boldsymbol{x}\in\mathbb{H}^n$, then we say $\omega$ is a left $\boldsymbol{i}$-eigenvalue of $A$ and $\boldsymbol{x}$ is a left $\boldsymbol{i}$-eigenvector of $A$ with respect to $\omega$. Similarly, if $A\overline{\boldsymbol{y}}^{\boldsymbol{i}}=\boldsymbol{y}\lambda$ holds for a quaternion $\lambda$ and a nonzero vector $\boldsymbol{y}\in\mathbb{H}^n$, then we say $\lambda$ is a right $\boldsymbol{i}$-eigenvalue of $A$ and $\boldsymbol{y}$ is a right $\boldsymbol{i}$-eigenvector of $A$ with respect to $\lambda$.
\end{definition}

From the definition of right$\backslash$left $\boldsymbol{i}$-eigenvalues we realize that a quaternion $\lambda$ is a right(left) $\boldsymbol{i}$-eigenvalue of $A$ if and only if $\overline{\lambda}^{\boldsymbol{i}}$ is a right$\backslash$left $\boldsymbol{i}$-eigenvalue of $\overline{A}^{\boldsymbol{i}}$.

\begin{proposition}\label{p3.3}
Let $A$ be in $M_n(\mathbb{H})$ and $\lambda$ is a right $\boldsymbol{i}$-eigenvalue of $A$. Then all the quaternions $\boldsymbol{i}$-similar to $\lambda$ are right $\boldsymbol{i}$-eigenvalues of $A$.
\end{proposition}

  \begin{proof}
    Let $\boldsymbol{x}$ be a right $\boldsymbol{i}$-eigenvector of $A$ with respect to the right $\boldsymbol{i}$-eigenvalue $\lambda$, then we have
    $A\overline{\boldsymbol{x}}^{\boldsymbol{i}}=\boldsymbol{x}\lambda$. For every nonzero $p\in\mathbb{H}$, we have
    $$A\overline{\boldsymbol{x}}^{\boldsymbol{i}}(\overline{p}^{\boldsymbol{i}})^{-1}=
    \boldsymbol{x}\lambda (\overline{p}^{\boldsymbol{i}})^{-1}=
    \boldsymbol{x}p^{-1}(p\lambda (\overline{p}^{\boldsymbol{i}})^{-1}),$$
    and we can know that $p\lambda (\overline{p}^{\boldsymbol{i}})^{-1}$ is the right $\boldsymbol{i}$-eigenvalue of $A$.
  \end{proof}

\begin{remark}
The conclusion in Proposition \ref{p3.3} does not hold for the left $\boldsymbol{i}$-eigenvalues of nonzero quaternion matrices.
\end{remark}

\begin{proposition}
Let $A,B$ be two $\boldsymbol{i}$-similar matrices in $M_n(\mathbb{H})$. Then $A$ and $B$ have the same right
$\boldsymbol{i}$-eigenvalues.
\end{proposition}

    \begin{proof}
    Since $A \sim_{\boldsymbol{i}} B$, then there exists an invertible matrix $P \in M_n(\mathbb{H})$ such that $\overline{P}^{\boldsymbol{i}}AP^{-1}=B$. Assume $\lambda$ is a right $\boldsymbol{i}$-eigenvalue of $A$ and $\boldsymbol{x}$ is the corresponding right $\boldsymbol{i}$-eigenvector of $A$, then
    $$
    A \overline{\boldsymbol{x}}^{\boldsymbol{i}} = \boldsymbol{x} \lambda \Leftrightarrow \overline{P}^{\boldsymbol{i}}A \overline{\boldsymbol{x}}^{\boldsymbol{i}} = \overline{P}^{\boldsymbol{i}}\boldsymbol{x} \lambda \Leftrightarrow BP\overline{\boldsymbol{x}}^{\boldsymbol{i}}=\overline{P}^{\boldsymbol{i}}\boldsymbol{x} \lambda,
    $$
which shows that $\lambda$ is the right $\boldsymbol{i}$-eigenvalue of $B$ and $\overline{P}^{\boldsymbol{i}}\boldsymbol{x}$ is the corresponded right $\boldsymbol{i}$-eigenvector of $B$, hence $A$ and $B$ have the same right $\boldsymbol{i}$-eigenvalues.
    \end{proof}

\begin{example}
Like the right eigenvalues of a quaternion matrix $A$, having the same right $\boldsymbol{i}$-eigenvalues is a necessary but not sufficient condition for $\boldsymbol{i}$-similarity.

Let $A=
\begin{pmatrix}
0 & 1 \\
0 & 0 \\
\end{pmatrix}$, $B=
\begin{pmatrix}
0 & 0 \\
0 & 0 \\
\end{pmatrix}$, then $\lambda=0$ is the only right $\boldsymbol{i}$-eigenvalue of both $A$ and $B$. However, $A$ is not $\boldsymbol{i}$-similar to $B$ because if we assume $P\in M_{2}(\mathbb{H})$ such that $\overline{P}^{\boldsymbol{i}}A=BP$, then a simple calculation implies that $P$ is not invertible, hence $A$ and $B$ are not $\boldsymbol{i}$-similar in $M_{2}(\mathbb{H})$.
\end{example}

\begin{lemma}\label{miemie}
Let $A\in M_{n}(\mathbb{H})$ and $\lambda\in \mathbb{H}$. Then the following statements hold.

(1) $\lambda$ is a right $\boldsymbol{i}$-eigenvalue of $A$ if and only if $\boldsymbol{i}\lambda$ is a right eigenvalue of $\boldsymbol{i}A$ $\Leftrightarrow$ $\lambda\boldsymbol{i}$ is a right eigenvalue of $A\boldsymbol{i}$.

(2) $\lambda$ is a left $\boldsymbol{i}$-eigenvalue of $A$ if and only if $\boldsymbol{i}\lambda$ is a left $\boldsymbol{i}$-eigenvalue of $\boldsymbol{i}A$ $\Leftrightarrow$ $\lambda\boldsymbol{i}$ is a left $\boldsymbol{i}$-eigenvalue of $A\boldsymbol{i}$.
\end{lemma}
\begin{proof}
The proof is simple. Moreover, it should be noted that if $A$ is a matrix in $M_{n}(\mathbb{H})$, then the right $\boldsymbol{i}$-eigenvectors of $A$ with respect to non $\boldsymbol{i}$-similar right $\boldsymbol{i}$-eigenvalues are right linearly independent and  the right $\boldsymbol{i}$-eigenvalues $A$ have at most n $\boldsymbol{i}$-similar classes. As for the proof of above truth, one can refer to \cite{H}.
\end{proof}

Equations of the type $ax+b=xc$ have been considered in \cite{RE1944}. It is shown there
if $a, b, c$ are quaternions and  $a$ is not similar to $c$, then $ax+b=xc$ has a
solution in $\mathbb{H}$.

\begin{lemma}\label{815}
Let $a, b, c$ be quaternions. If $a$ is not $\boldsymbol{i}$-similar to $c$, then the equation $ax+b=\overline{x}^{\boldsymbol{i}}c$ has a solution in $\mathbb{H}$.
\end{lemma}
\begin{proof}
Notice that $a\sim_{\boldsymbol{i}}c$ if and only if $\boldsymbol{i}a\sim\boldsymbol{i}c$. Assume $\boldsymbol{i}a$ is not similar to $\boldsymbol{i}c$, then the equation
\begin{equation}\label{12261}
(\boldsymbol{i}a)x+d=x(\boldsymbol{i}c)
\end{equation}
has a solution $x_{0}$ for any quaternion constant d. Let $d=\boldsymbol{i}b$, then $d$ is a quaternion constant and equation (\ref{12261}) becomes
\begin{equation}\label{12262}
(\boldsymbol{i}a)x+\boldsymbol{i}b=x(\boldsymbol{i}c)
\end{equation}
Substituting $x_{0}$ into (\ref{12262}), we obtain $(\boldsymbol{i}a)x_{0}+\boldsymbol{i}b=x_{0}(\boldsymbol{i}c)$,
which is equivalent to  $ax_{0}+b=\overline{x_0}^{\boldsymbol{i}}c$. Therefore, $x_{0}$ is a quaternion solution of the equation $ax+b=\overline{x}^{\boldsymbol{i}}c$.
\end{proof}

For a quaternion matrix $A$, if $A$ is in triangular form, then every diagonal element is a right eigenvalue of $A$. In fact, if $A$ is in triangular form, then every diagonal element is a right $\boldsymbol{i}$-eigenvalue of $A$. The proof is analogous to Theorem $11$ in \cite{B} and Lemma \ref{815} may be effective for the proof.

\begin{proposition}
Let $A$ be a quaternion matrix of order $n$. If $A$ is of upper(lower) triangular form, then all the right $\boldsymbol{i}$-eigenvalues of $A$ are exactly the diagonal elements and all the quaternions $\boldsymbol{i}$-similar to them.
\end{proposition}

\subsection{i-Jordan canonical form}
The Jordan canonical form of a matrix plays a significant role in matrix operations. To a large extent, it simplifies the operation and analysis of matrices, especially in dealing with non-diagonalizable matrices, solving linear differential equations and control theory. Many mathematicians have also conducted research on the Jordan canonical form of quaternion matrices and drawn a series of conclusions. The Jordan canonical form of a quaternion matrix is as follows.
\begin{lemma}[\cite{H}]\label{H}
Let $A$ be in $M_n(\mathbb{H})$. Then
$$A \sim J_A:= J_{n_1}(\lambda_1) \oplus J_{n_2}(\lambda_2) \oplus \cdots \oplus J_{n_r}(\lambda_r),$$
where $\lambda_k=a_k+b_k\boldsymbol{i}\in\mathbb{C}$ are right eigenvalues of $A(k=1, \cdots, r)$ with $a_k, \ b_k\in\mathbb{R}$ may be chosen so that $b_k\ge0$. $J_A$ is uniquely determined by $A$ up to the order of Jordan blocks $J_{n_k}(\lambda_k)$.
\end{lemma}

Similarly, we obtain the Jordan canonical form of a quaternion matrix in the $\boldsymbol{i}$-similar sense.
\begin{theorem}\label{ISJCF}
Let $A\in M_n(\mathbb{H})$. Then
$$A \sim_{\boldsymbol{i}} J_A^{\boldsymbol{i}}:= J_{n_1}(\lambda_1) \oplus J_{n_2}(\lambda_2) \oplus \cdots \oplus J_{n_r}(\lambda_r),$$
where $\lambda_{k}=-a_k+b_k\boldsymbol{i}$ are complex right $\boldsymbol{i}$-eigenvalues of $A$ where $a_k,b_k\in\mathbb{R}$ may be chosen so that $a_k\ge0$. We call $J_A^{\boldsymbol{i}}$ the $\boldsymbol{i}$-Jordan canonical form of $A$ and $J_{n_k}(\lambda_k)$ the $\boldsymbol{i}$-Jordan blocks, hence $J_A^{\boldsymbol{i}}$ is uniquely determined by $A$ up to the order of $\boldsymbol{i}$-Jordan blocks $J_{n_k}(\lambda_k)$.

\end{theorem}

    \begin{proof}
By Lemma \ref{H}, there exists an invertible matrix
$P\in M_n(\mathbb{H})$ such that
$$
 P^{-1}(\boldsymbol{i}A)P
 =J_{\boldsymbol{i}A}
 =\bigoplus_{\ell=1}^{r}
   J_{n_\ell}(-b_\ell+a_\ell\boldsymbol{i}),
 \qquad
 a_\ell\geq 0,\quad b_\ell\in\mathbb R.
$$
Set
$$
 \lambda_\ell=-a_\ell+b_\ell\boldsymbol{i},
 \qquad
 J_A^{(\boldsymbol{i})}
 =\bigoplus_{\ell=1}^{r}J_{n_\ell}(\lambda_\ell).
$$

We now correct the transformation of the superdiagonal entries.
For each positive integer $m$, let
\[
 N_m=J_m(0),\qquad
 q=\boldsymbol{j}+\boldsymbol{k},\qquad
 Q_m=qI_m,
\]
and define
\[
 D_m=\operatorname{diag}
 \bigl(1,\boldsymbol{i},\boldsymbol{i}^{\,2},
       \ldots,\boldsymbol{i}^{\,m-1}\bigr).
\]
Since
\[
 q^{-1}\boldsymbol{i}q=-\boldsymbol{i}
\]
and
\[
 D_m^{-1}N_mD_m=\boldsymbol{i}N_m,
\]
setting $S_m=Q_mD_m$ gives
\begin{align*}
 S_m^{-1}J_m(-b+a\boldsymbol{i})S_m
 &=D_m^{-1}
   \bigl((-b-a\boldsymbol{i})I_m+N_m\bigr)D_m\\
 &=(-b-a\boldsymbol{i})I_m+\boldsymbol{i}N_m\\
 &=\boldsymbol{i}
   \bigl((-a+b\boldsymbol{i})I_m+N_m\bigr)\\
 &=\boldsymbol{i}J_m(-a+b\boldsymbol{i}).
\end{align*}

Let
\[
 S=\bigoplus_{\ell=1}^{r}S_{n_\ell},
 \qquad
 R=PS.
\]
It follows that
\[
 R^{-1}(\boldsymbol{i}A)R
 =S^{-1}J_{\boldsymbol{i}A}S
 =\boldsymbol{i}J_A^{(\boldsymbol{i})}.
\]
By Lemma \ref{RIST}, this is equivalent to
\[
 \bigl(R^{\boldsymbol{i}}\bigr)^{-1}AR
 =J_A^{(\boldsymbol{i})},
\]
where
\[
 R^{\boldsymbol{i}}
 =-\boldsymbol{i}R\boldsymbol{i}.
\]
Therefore,
\[
 A\sim_{\boldsymbol{i}}J_A^{(\boldsymbol{i})}.
\]

It remains to verify the statement concerning the right
$\boldsymbol{i}$-eigenvalues. Since
\[
 q^{-1}(-b_\ell+a_\ell\boldsymbol{i})q
 =-b_\ell-a_\ell\boldsymbol{i}
 =\boldsymbol{i}\lambda_\ell,
\]
the quaternion $\boldsymbol{i}\lambda_\ell$ is a right eigenvalue
of $\boldsymbol{i}A$. Hence, by Lemma \ref{miemie}, $\lambda_\ell$ is a
right $\boldsymbol{i}$-eigenvalue of $A$.

Finally, the Jordan blocks of $J_{\boldsymbol{i}A}$ are uniquely
determined up to their order by Lemma \ref{H}. The correspondence
\[
 (n_\ell,-b_\ell+a_\ell\boldsymbol{i})
 \longmapsto
 (n_\ell,-a_\ell+b_\ell\boldsymbol{i})
\]
is one-to-one. Consequently, $J_A^{(\boldsymbol{i})}$ is uniquely
determined up to the order of its Jordan blocks.
\end{proof}
    
Following from Theorem \ref{ISJCF}, one can see that if we assume $A$ is in $M_n(\mathbb{H})$, $A$ is $\boldsymbol{i}$-similar to a complex matrix and the following are equivalent.

(1) The $\boldsymbol{i}$-Jordan canonical form of $A$ is
$$J_A^{\boldsymbol{i}} = J_{n_1}(-a_{1}+b_{1}\boldsymbol{i}) \oplus J_{n_2}(-a_{2}+b_{2}\boldsymbol{i}) \oplus \cdots \oplus J_{n_r}(-a_{r}+b_{r}\boldsymbol{i}),$$
where $a_{k}, b_{k}\in \mathbb{R}$ and $a_{k}\geq 0, k=1,\cdots, r$.

(2) The Jordan canonical form of ${\boldsymbol{i}}A$ under similarity is
$$J_{\boldsymbol{i}A} = J_{n_1}(-b_{1}+a_{1}\boldsymbol{i}) \oplus J_{n_2}(-b_{2}+a_{2}\boldsymbol{i}) \oplus \cdots \oplus J_{n_r}(-b_{r}+a_{r}\boldsymbol{i}),$$
where $a_{k}, b_{k}\in \mathbb{R}$ and $a_{k}\geq 0, k=1,\cdots, r$.

\begin{corollary}\label{ISJCF2}
Let $A,B$ be two  matrices in $M_n(\mathbb{H})$, then
\begin{enumerate}
\item $A \overline{A}^{\boldsymbol{i}}\sim \overline{A}^{\boldsymbol{i}}A$.
\item $A \sim_{\boldsymbol{i}} B$ if and only if $A$ and $B$ have the same $\boldsymbol{i}$-Jordan canonical form.
\end{enumerate}
\end{corollary}

    \begin{proof}
    \begin{enumerate}
   \item  By Theorem \ref{ISJCF}, let $A=\overline{P}^{\boldsymbol{i}}J_A^{\boldsymbol{i}}P^{-1}$, then $\overline{A}^{\boldsymbol{i}}=P\overline{J_A^{\boldsymbol{i}}}^{\boldsymbol{i}}\overline{P^{-1}}^{\boldsymbol{i}}$.
     we obtain $A\overline{A}^{\boldsymbol{i}} \sim J_A^{\boldsymbol{i}}\overline{J_A^{\boldsymbol{i}}}^{\boldsymbol{i}}$ and $\overline{A}^{\boldsymbol{i}}A \sim \overline{J_A^{\boldsymbol{i}}}^{\boldsymbol{i}}J_A^{\boldsymbol{i}}$. Since $J_A^{\boldsymbol{i}}\overline{J_A^{\boldsymbol{i}}}^{\boldsymbol{i}}=\overline{J_A^{\boldsymbol{i}}}^{\boldsymbol{i}}J_A^{\boldsymbol{i}}
    =(J_A^{\boldsymbol{i}})^2$, we have $A\overline{A}^{\boldsymbol{i}} \sim \overline{A}^{\boldsymbol{i}} A$.
    \item The conclusion $2$ is a straightforward corollary of Theorem \ref{ISJCF}.
\end{enumerate}
    \end{proof}
By Theorem \ref{ISJCF} and Corollary \ref{ISJCF2}, we give an affirmative answer to Kaplansky's test problem in the sense of $\boldsymbol{i}$-similar equivalence.

\begin{corollary}
Let $A, B, C$ be in $M_n(\mathbb{H})$, then

\begin{enumerate}
\item $\begin{pmatrix}
A & 0 \\
0 & B
\end{pmatrix}
\sim_{\boldsymbol{i}}
\begin{pmatrix}
A & 0 \\
0 & C
\end{pmatrix}$ if and only if $B \sim_{\boldsymbol{i}} C$,

\item $\begin{pmatrix}
B & 0 \\
0 & B
\end{pmatrix}
\sim_{\boldsymbol{i}}
\begin{pmatrix}
C & 0 \\
0 & C
\end{pmatrix}$ if and only if $B \sim_{\boldsymbol{i}} C$.
\end{enumerate}
\end{corollary}
\begin{corollary}
If the right $\boldsymbol{i}$-eigenvalues of a quaternion matrix $A$ have exactly n $\boldsymbol{i}$-similar classes, then $A$ can be transformed into a diagonal matrix by $\boldsymbol{i}$-similarity.
\end{corollary}

\section{Unitary i-congruence}
In this section, we investigate unitary \(\boldsymbol{i}\)-congruence of
quaternion matrices. We characterize quaternion matrices that
are unitarily \(\boldsymbol{i}\)-congruent to complex matrices, study
\(\boldsymbol{i}\)-conjugate normality and triangularization under unitary
\(\boldsymbol{i}\)-congruence, and finally introduce the \(\boldsymbol{i}\)-polar
decomposition.
\subsection{Quaternion matrices unitary i-congruent to complex matrices}

We propose the necessary and sufficient conditions for a quaternion matrix being unitarily $\boldsymbol{i}$-congruent to a complex matrix.
\begin{theorem}\label{10242}
A matrix $A\in M_{n}(\mathbb{H})$ is unitarily $\boldsymbol{i}$-congruent to a complex matrix if and only if $A$ is unitarily $\boldsymbol{i}$-congruent to the $\boldsymbol{i}$-conjugate $\overline{A}^{\boldsymbol{i}}$ by a special unitary matrix $P$ with $P=UU^{*\boldsymbol{i}}$, where $U\in M_{n}(\mathbb{H})$ is a unitary matrix.
\end{theorem}
\begin{proof}
The proof of this conclusion ia similar to that of Theorem \ref{QC1}, with the only difference being that we need to replace unitary equivalence with unitary $\boldsymbol{i}$-congruence.
\end{proof}

\begin{corollary}\label{10243}
Let $A\in M_{n}(\mathbb{H})$, then $A$ is unitarily $\boldsymbol{i}$-congruent to a real matrix in $M_{n}(\mathbb{H})$ if and only if $A$ is unitarily  $\boldsymbol{i}$-congruent to a complex matrix $C$ in $M_{n}(\mathbb{H})$ and the complex matrix $C$ is unitarily  $\boldsymbol{i}$-congruent(or, unitarily equivalent) to a real matrix in $M_{n}(\mathbb{C})$.
\end{corollary}

\subsection{i-conjugate normality}
As well known, each matrix $A\in M_{n}(\mathbb{H})$ can be unitarily equivalent to an upper triangular matrix $T$ is well-known to everyone.
\begin{theorem}[\cite{B}]\label{yy}
Every matrix $A$ of real quaternions can be transformed into a triangular form $T$ by a unitary matrix $U$.
\end{theorem}

For $A\in M_{n}(\mathbb{H})$, we find that the unitary $\boldsymbol{i}$-congruence could do the same thing as unitary equivalence does in Theorem \ref{yy}.
\begin{theorem}\label{QMISUT}
Let $A$ be a matrix in $M_n(\mathbb{H})$, then $A$ is unitarily  $\boldsymbol{i}$-congruent to an upper triangular matrix, that is, there exists a unitary matrix $U\in M_n(\mathbb{H})$ such that $UAU^{*\boldsymbol{i}}$ is an upper triangular matrix.
\end{theorem}
\begin{proof}
Assume the quaternion matrix $\boldsymbol{i}A$ is unitarily equivalent to an upper triangular matrix $\boldsymbol{i}T$, hence $T$ is itself upper triangular. Note that the  unitary equivalence between $\boldsymbol{i}A$ and $\boldsymbol{i}T$ is equivalent to the unitary $\boldsymbol{i}$-congruence between $A$ and $T$, we obtain the conclusion.
\end{proof}
Especially, if $A$ is normal, then $T$ in Theorem \ref{yy} is diagonal, if $A$ is Hermitian, all the diagonal elements of $T$ can be real numbers.
\begin{corollary}[\cite{Z}]
A matrix $A$ in $M_n(\mathbb{H})$ is normal if and only if there exists a unitary matrix $U$ such that
$$
U^{*}AU=\mathrm{diag} (h_1+k_1\boldsymbol{i},\cdots,h_n+k_n\boldsymbol{i}),
$$
where $h_s$ and $k_s$ are real numbers for $1\leq s\leq n$, the $k_s$ can be taken as nonnegative numbers actually, and $A$ is Hermitian if and only if $k_1=k_2=\cdots=k_n=0$.
\end{corollary}

A normal matrix can be transformed into a diagonal matrix through unitary equivalence, we show a new concept analogous to normality to achieve diagonalization under unitary $\boldsymbol{i}$-congruence.

\begin{definition}
Let $A$ be a matrix in $M_n(\mathbb{H})$, if $A$ satisfies
$$\overline{A^*A}^{\boldsymbol{i}}=\overline{A^*}^{\boldsymbol{i}}\overline{A}^{\boldsymbol{i}}=AA^*,$$
then we say $A$ is an $\boldsymbol{i}$-conjugate normal matrix.
\end{definition}

\begin{corollary}
Let $A$ be a  matrix in $M_n(\mathbb{H})$. Then the $\boldsymbol{i}$-conjugate normality of $A$  is equivalent to the $\boldsymbol{i}$-conjugate normality of $A^*, \overline{A}^{\boldsymbol{i}},\ A^{*{\boldsymbol{i}}},\ A^{-1}$(if $A$ is invertible), respectively.
\end{corollary}

\begin{theorem}\label{RUISIN}
Let $A$ be in $M_n(\mathbb{H})$, then the following are equivalent.
\begin{enumerate}
\item $A$ is unitarily $\boldsymbol{i}$-congruent to a diagonal matrix,

\item $A$ is an $\boldsymbol{i}$-conjugate normal matrix.
\end{enumerate}
\end{theorem}

 \begin{proof}
 We can draw the conclusion of the Theorem from two perspectives, one is, the proof of this Theorem is analogous to that of the well-known result in quaternion matrix theory that a quaternion matrix $A$ is unitarily equivalent to a diagonal matrix if and only if $A$ is a normal matrix. Another is, one should note that $A$ is $\boldsymbol{i}$-conjugate normal if and only if $\boldsymbol{i}A$ is normal and $\overline{U}^{\boldsymbol{i}}AU^{*}=T$ for a unitary $U$ and a diagonal $T$ is equivalent to $U(\boldsymbol{i}A)U^{*}=\boldsymbol{i}T$.
 \end{proof}

In fact, by Lemma \ref{yingwen}, Proposition \ref{p3.3} and the truth that the diagonal elements of a diagonal matrix are its right $\boldsymbol{i}$-eigenvalues, the main diagonal elements of $T$ can be taken as complex numbers with nonnegative real parts, we call such a $T$ the canonical form of an $\boldsymbol{i}$-conjugate normal matrix $A$ under unitary $\boldsymbol{i}$-congruence. We can see that $T$ is normal actually.

\begin{definition}\label{10221}
Let $A$ be a matrix in $M_n(\mathbb{H})$, if $\overline{ A }^{\boldsymbol{i}}= A^* $ ($\overline{ A }^{\boldsymbol{i}} = -A^{*}$), or equivalently, $A=A^{*{\boldsymbol{i}}}$ ($A=-A^{*{\boldsymbol{i}}}$), then we say $A$ is an $\boldsymbol{i}$-Hermitian matrix (skew $\boldsymbol{i}$-Hermitian matrix).
\end{definition}
From Definition \ref{10221} we can see that both $\boldsymbol{i}$-Hermitian and skew $\boldsymbol{i}$-Hermitian quaternion matrices are $\boldsymbol{i}$-conjugate normal matrices.
\begin{corollary}
Let $A\in M_{n}(\mathbb{H})$, then the $\boldsymbol{i}$-Hermiticity(skew $\boldsymbol{i}$-Hermiticity) of $A$ is equivalent to the $\boldsymbol{i}$-Hermiticity(skew $\boldsymbol{i}$-Hermiticity) of $A^{*},\ \overline{A}^{\boldsymbol{i}},\ A^{*{\boldsymbol{i}}}$, $\overline{A},\ A^{T},\ A^{-1}$ (if A is invertible), respectively.
\end{corollary}

\begin{corollary}\label{111}
Let $A$ be a matrix in $M_n(\mathbb{H})$. If $A$ is an $\boldsymbol{i}$-Hermitian matrix, then all the right $\boldsymbol{i}$-eigenvalues of $A$ can be $\boldsymbol{i}$-similar to real numbers. If $A$ is a skew $\boldsymbol{i}$-Hermitian matrix, then all the right $\boldsymbol{i}$-eigenvalues of $A$ are zero or have the form $t\boldsymbol{i}$ for some real number $t$.
\end{corollary}

\begin{corollary}\label{916}
Let $A$ be an $\boldsymbol{i}$-conjugate normal matrix in $M_n(\mathbb{H})$. If $\lambda$ is the right $\boldsymbol{i}$-eigenvalue of $A$ and $\boldsymbol{x}$ is the right $\boldsymbol{i}$-eigenvector of $A$ with respect to $\lambda$, then $\overline{ \lambda }^{\boldsymbol{i}}$ and $\overline{ \lambda }$ are right $\boldsymbol{i}$-eigenvalues of $\overline{ A }^{\boldsymbol{i}}$ and $A^*$ respectively, $\overline{ \boldsymbol{x} }^{\boldsymbol{i}}$ is a right $\boldsymbol{i}$-eigenvector of $\overline{ A }^{\boldsymbol{i}}$ and $A^*$ with respect to $\overline{ \lambda }^{\boldsymbol{i}}$ and $\overline{\lambda}$, respectively.
\end{corollary}
\begin{corollary}
Let $A\in M_{n}(\mathbb{H})$ be an $\boldsymbol{i}$-conjugate normal matrix, suppose that $\lambda$ and $\mu$ are right $\boldsymbol{i}$-eigenvalues of $A$ and $\lambda$ is not $\boldsymbol{i}$-similar to $\mu$. If $\boldsymbol{x}$ and $\boldsymbol{y}$ are right $\boldsymbol{i}$-eigenvectors of $A$ with respect to $\lambda$ and $\mu$, that is,
$$
A\overline{\boldsymbol{x}}^{\boldsymbol{i}}=\boldsymbol{x}\lambda,\ \ A\overline{\boldsymbol{y}}^{\boldsymbol{i}}=\boldsymbol{y}\mu,
$$
then the vectors $\boldsymbol{x}$ and $\boldsymbol{y}$ are orthogonal.
\end{corollary}
\begin{proof}
By Corollary \ref{916}, one can see that
$$
A^{*\boldsymbol{i}}\overline{\boldsymbol{y}}^{\boldsymbol{i}}=\boldsymbol{y}\overline{\overline{\mu}}^{\boldsymbol{i}},
$$
then
$$
(\boldsymbol{y}^{*}\boldsymbol{x})\lambda=\boldsymbol{y}^{*}(A\overline{\boldsymbol{x}}^{\boldsymbol{i}})=(A^{*\boldsymbol{i}}\overline{\boldsymbol{y}}^{\boldsymbol{i}})^{*\boldsymbol{i}}\overline{\boldsymbol{x}}^{\boldsymbol{i}}
=(\boldsymbol{y}\overline{\overline{\mu}}^{\boldsymbol{i}})^{*\boldsymbol{i}}\overline{\boldsymbol{x}}^{\boldsymbol{i}}
=\mu(\boldsymbol{y}^{*\boldsymbol{i}}\overline{\boldsymbol{x}}^{\boldsymbol{i}}).
$$
Suppose that  $\boldsymbol{x}$ and $\boldsymbol{y}$ are not orthogonal, that is, $\boldsymbol{y}^{*}\boldsymbol{x}\neq 0$, then we deduce $\lambda$ is $\boldsymbol{i}$-similar to $\mu$ by the equation above, which is in contradiction to our assumption. Hence
$\langle \boldsymbol{x},\boldsymbol{y} \rangle=\boldsymbol{y}^{*}\boldsymbol{x}=0$.
\end{proof}

We obtain a decomposition of an $\boldsymbol{i}$-Hermitian matrix in $M_n(\mathbb{H})$.
\begin{theorem}\label{11251}
Let $A\in M_{n}(\mathbb{H})$, then $A$ is an $\boldsymbol{i}$-Hermitian matrix in $M_n(\mathbb{H})$ if and only if there exists a matrix $P \in M_n(\mathbb{H})$ such that
$A = \overline{ P }^{\boldsymbol{i}}  P^*$.
\end{theorem}

\begin{proof}
 On the one hand, if $A = \overline{ P }^{\boldsymbol{i}}  P^*$, it is easy to check that $\overline{ A}^{\boldsymbol{i}} =A^{*}$, so $A$ is $\boldsymbol{i}$-Hermitian.

   On the other hand, since $A$ is an $\boldsymbol{i}$-Hermitian matrix, then $A$ is unitarily $\boldsymbol{i}$-congruent to a diagonal matrix by Theorem \ref{RUISIN}, that is
    $U^{*}A\overline{ U }^{\boldsymbol{i}} = \mathrm{diag} (\lambda_1, \cdots, \lambda_n)$, where $\lambda_s$ is real number, especially, non-negative real number. Set $t_k=\sqrt{\lambda_k}$. Let $P = \overline{ U }^{\boldsymbol{i}} \mathrm{diag} (t_1, \cdots, t_n)$, then $\overline{ P }^{\boldsymbol{i}} P^* = U \mathrm{diag} (\lambda_1, \cdots, \lambda_n) (\overline{ U }^{\boldsymbol{i}})^{*} = A$.
 \end{proof}
\begin{remark}
Following from Theorem \ref{11251}, we know that the unitary matrices $P$ in Theorem \ref{QC1} and Theorem \ref{10242} are $\boldsymbol{i}$-Hermitian.
\end{remark}
\subsection{Quaternion matrix triangularization under unitary i-congruence}
In this section, we provide a decomposition of a class of special quaternion matrices. The main Theorem is Theorem \ref{3.3}, and the quaternion Takagi's factorization is obtained as a corollary of Theorem \ref{3.3}, which has been studied by F. Zhang in \cite{HZ2012}. Firstly, we propose some known results that will be used in the proof of our conclusion.
\begin{theorem}[\cite{B}]\label{6.10}
If $A\in M_n(\mathbb{H})$ is in triangular form, then every diagonal element is a (right) eigenvalue of $A$. Conversely, every (right) eigenvalue of $A$
is similar to a diagonal entry of $A$.
\end{theorem}

\begin{theorem}[\cite{Z}]\label{6.11}
Let $u_1$ be a unit column vector of $n$ quaternion components. Then there exist $n-1$ unit column vectors $u_2,\cdots,u_n$ of $n$ column
components such that $\{u_1, u_2, \cdots, u_n\}$ is an orthogonal set, i.e., $u_s^*u_t=0,\ s\ne t$.
\end{theorem}

\begin{theorem}[\cite{HR}]\label{wuwu}
Let $A\in M_{n}({\mathbb{C}})$ be given. Then there exists a unitary $U\in M_{n}({\mathbb{C}})$ and an upper triangular $\Delta\in M_{n}({\mathbb{C}})$ such that $A=U\Delta U^{T}$ if and only if all the eigenvalues of $A\overline{A}$ are real and nonnegative. Under this condition, all the main diagonal entries of $\Delta$ may be chosen to be nonnegative.
\end{theorem}
The Lemma below is a special case of Theorem \ref{wuwu}.
\begin{lemma}[see Corollary 4.4.4 in \cite{HR}](Takagi's factorization)\label{12291}
If $A\in M_{n}({\mathbb{C}})$ is symmetric, then there exists a unitary $U\in M_{n}({\mathbb{C}})$ and a real nonnegative diagonal matrix $\Sigma$ such that $A=U\Sigma U^{T}$.
\end{lemma}
We will prove that for matrices in $M_{n}(\mathbb{H})$, the similar conclusions of Theorem \ref{wuwu} and Lemma \ref{12291} still hold. First we show two lemmas that are used in the proof.
\begin{lemma}\label{jin}
Let $q$ be a quaternion. Then $q\overline{q}^{\boldsymbol{i}}=s$ for some real nonnegative number $s$ if and only if q can be represented as $q=q_{0}+q_{2}\boldsymbol{j}+q_{3}\boldsymbol{k}$ where $q_{0}, q_{2}, q_{3}\in \mathbb{R}$, that is, the coefficient of $\boldsymbol{i}$ in $q$ is missing.
\end{lemma}
\begin{proof}

Let $q=a_{1}+\boldsymbol{j}a_{2}$, $a_{1}, a_{2}\in\mathbb{C}$, then $\overline{q}^{\boldsymbol{i}}=a_{1}-\boldsymbol{j}a_{2}$. Assume
$$
q\overline{q}^{\boldsymbol{i}}=a_{1}^{2}+|a_{2}|^{2}+\boldsymbol{j}(a_{2}a_{1}-\overline{a_{1}}a_{2})=s\geq 0.
$$
\begin{enumerate}
\item [1.] If $a_{2}=0$, then $a_{1}^{2}=s$, that is, $a_{1}\in \mathbb{R}$. In this case, $q$ has the desired form.
\item [2.] If $a_{2}\neq 0$, then $a_{1}\in \mathbb{R}$ and $a_{1}^{2}+|a_{2}|^{2}=s$. In this case, $q$ has the desired form either.
\end{enumerate}
The proof of another direction is easy.
\end{proof}

\begin{lemma}\label{jintian}
If a quaternion $q$  satisfies the condition that  $q\overline{q}^{\boldsymbol{i}}=s$ for some real nonnegative number $s$, then there exists a quaternion $p$ such that $pq(\overline{p}^{\boldsymbol{i}})^{-1}=\sqrt{s}$, where $s$ is a nonnegative real number.
\end{lemma}
\begin{proof}
By Lemma \ref{jin}, we may write
\[
q=a_1+\boldsymbol{j}a_2,
\qquad
a_1\in\mathbb{R},\quad a_2\in\mathbb{C}.
\]
Since
\[
\overline{q}^{\boldsymbol{i}}=a_1-\boldsymbol{j}a_2,
\]
we have
\[
q\overline{q}^{\boldsymbol{i}}
 =a_1^2+|a_2|^2=s.
\]
Put
\[
r=\sqrt{s}\geq 0.
\]
We shall construct a nonzero quaternion \(p\) such that
\[
pq=r\,\overline{p}^{\boldsymbol{i}}.
\]
This is equivalent to
\[
pq\bigl(\overline{p}^{\boldsymbol{i}}\bigr)^{-1}=r.
\]

First suppose that \(s=0\). Then
\[
a_1^2+|a_2|^2=0,
\]
so \(a_1=0\), \(a_2=0\), and hence \(q=0\). Taking \(p=1\), we obtain
\[
pq\bigl(\overline{p}^{\boldsymbol{i}}\bigr)^{-1}=0=\sqrt{s}.
\]

Now suppose that \(s>0\).

If \(a_2=0\), then \(q=a_1\in\mathbb{R}\) and
\[
a_1^2=s=r^2.
\]
Thus \(a_1=r\) or \(a_1=-r\).

If \(a_1=r\), take \(p=1\). Then
\[
pq\bigl(\overline{p}^{\boldsymbol{i}}\bigr)^{-1}
=q=r.
\]

If \(a_1=-r\), take \(p=\boldsymbol{j}\). Since
\[
\overline{p}^{\boldsymbol{i}}=-\boldsymbol{j},
\qquad
\bigl(\overline{p}^{\boldsymbol{i}}\bigr)^{-1}=\boldsymbol{j},
\]
we obtain
\[
pq\bigl(\overline{p}^{\boldsymbol{i}}\bigr)^{-1}
=\boldsymbol{j}(-r)\boldsymbol{j}=r.
\]

It remains to consider the case \(a_2\neq0\). Since
\[
r^2=a_1^2+|a_2|^2,
\]
we have \(r>|a_1|\), and in particular \(a_1-r\neq0\). Define
\[
p_2=1,
\qquad
p_1=\frac{a_2}{a_1-r},
\qquad
p=p_1+\boldsymbol{j}p_2=p_1+\boldsymbol{j}.
\]
Clearly \(p\neq0\). Using
\[
z\boldsymbol{j}=\boldsymbol{j}\overline z,
\qquad z\in\mathbb C,
\]
we calculate
\begin{align*}
pq
&=(p_1+\boldsymbol{j})(a_1+\boldsymbol{j}a_2)\\
&=a_1p_1-a_2
  +\boldsymbol{j}(\overline{p_1}a_2+a_1).
\end{align*}
For the complex component, we have
\begin{align*}
a_1p_1-a_2
&=\frac{a_1a_2}{a_1-r}-a_2\\
&=\frac{ra_2}{a_1-r}
 =rp_1.
\end{align*}
Since \(a_1-r\) is real, we have
\begin{align*}
\overline{p_1}a_2+a_1
&=\frac{|a_2|^2}{a_1-r}+a_1\\
&=\frac{|a_2|^2+a_1(a_1-r)}{a_1-r}\\
&=\frac{r^2-a_1r}{a_1-r}\\
&=-r.
\end{align*}
Consequently,
\[
pq=rp_1-\boldsymbol{j}r=r(p_1-\boldsymbol{j})=r\,\overline{p}^{\boldsymbol{i}}.
\]
Since \(p\neq0\), the quaternion \(\overline{p}^{\boldsymbol{i}}\) is invertible, and hence
\[
pq\bigl(\overline{p}^{\boldsymbol{i}}\bigr)^{-1}
=r=\sqrt{s}.
\]

Finally, \(p\) may be chosen to have unit modulus. Indeed, if \(p\) is
one of the nonzero quaternions constructed above and
\[
\widehat p=\frac{p}{|p|},
\]
then \(|\widehat p|=1\), and, since \(|p|\) is a positive real number,
\begin{align*}
\widehat p\,q
 \bigl(\overline{\widehat p}^{\boldsymbol{i}}\bigr)^{-1}
&=\frac{p}{|p|}\,q
  \left(\frac{\overline p^{\boldsymbol{i}}}{|p|}\right)^{-1}\\
&=pq\bigl(\overline p^{\boldsymbol{i}}\bigr)^{-1}\\
&=\sqrt{s}.
\end{align*}
This completes the proof.
\end{proof}

The following is our main result in this section.
\begin{theorem}\label{3.3}
Let $A$ be a  matrix in $M_n(\mathbb{H})$. Then the following two statements are equivalent.
\begin{enumerate}
\item [(a)] There exist a unitary matrix $U\in M_n(\mathbb{H})$ and an upper triangular matrix $\Delta\in M_n(\mathbb{H})$ such that 
$$
A=U\Delta U^{*\boldsymbol{i}},
$$
 where all the right eigenvalues of $\Delta\overline{\Delta}^{\boldsymbol{i}}$ are nonnegative real numbers.

\item [(b)] All the right eigenvalues of $A\overline{A}^{\boldsymbol{i}}$ are nonnegative real numbers.
\end{enumerate}
\end{theorem}

    \begin{proof}
Suppose first that $(a)$ holds. If there exists a unitary matrix $U\in M_n(\mathbb{H})$ and a desired upper triangular matrix $\Delta\in M_n(\mathbb{H})$ such
    that $A=U\Delta U^{*\boldsymbol{i}}$, then 
    $$
    A\overline{A}^{\boldsymbol{i}}=U\Delta U^{*\boldsymbol{i}}\overline{U}^{\boldsymbol{i}}\overline{\Delta}^{\boldsymbol{i}}U^*=U\Delta\overline{\Delta}^{\boldsymbol{i}}U^*.
    $$ 
    We can see that $A\overline{A}^{\boldsymbol{i}}$
    is unitarily equivalent to $\Delta\overline{\Delta}^{\boldsymbol{i}}$ and hence the right eigenvalues of $A\overline{A}^{\boldsymbol{i}}$ are the same as the right eigenvalues of the upper triangular matrix $\Delta\overline{\Delta}^{\boldsymbol{i}}$. Since the right eigenvalues of an upper triangular matrix are exactly its main diagonal entries by Theorem \ref{6.10}, we draw the conclusion that all the right eigenvalues of $A\overline{A}^{\boldsymbol{i}}$ are nonnegative real numbers.

    Conversely, we assume that $(b)$ holds. Assume that $A\overline{A}^{\boldsymbol{i}}$ has only nonnegative right eigenvalues, that is $A\overline{A}^{\boldsymbol{i}}\boldsymbol{x}=\boldsymbol{x}\lambda$ with $\lambda\ge0$ and
    $\boldsymbol{x}\ne\boldsymbol{0}$. There are two cases.

    (1) $A\overline{\boldsymbol{x}}^{\boldsymbol{i}}$ and $\boldsymbol{x}$ are right linearly dependent,

    (2) $A\overline{\boldsymbol{x}}^{\boldsymbol{i}}$ and $\boldsymbol{x}$ are right linearly independent.

    In case (1), there is some $\mu\in\mathbb{H}$ such that $A\overline{\boldsymbol{x}}^{\boldsymbol{i}}=\boldsymbol{x}\mu$. Then
    $$A\overline{A}^{\boldsymbol{i}}\boldsymbol{x}=A\overline{\boldsymbol{x}}^{\boldsymbol{i}}\overline{\mu}^{\boldsymbol{i}}
    =\boldsymbol{x}\mu \overline{ \mu }^{\boldsymbol{i}} =\boldsymbol{x}\lambda,$$
    hence $\mu \overline{ \mu }^{\boldsymbol{i}} =\lambda$. In particular, if $\lambda=0$, then $\mu=0$. One simply takes $\boldsymbol{v}=\boldsymbol{x}$ and $\alpha=\mu=0$. 

    In case (2), the vector $\boldsymbol{y}=A\overline{\boldsymbol{x}}^{\boldsymbol{i}}+\boldsymbol{x}\mu\ne \boldsymbol{0}$ for all $\mu\in\mathbb{H}$ and we can choose a nonnegative real number $\mu$
    such that $\lambda=\mu^2$. Then
    $$
     A\overline{\boldsymbol{{y}}}^{\boldsymbol{i}}=A(\overline{A}^{\boldsymbol{i}}\boldsymbol{x}+\overline{\boldsymbol{x}}^{\boldsymbol{i}}\overline{\mu}^{\boldsymbol{i}})
    =\boldsymbol{x}\lambda+A\overline{\boldsymbol{x}}^{\boldsymbol{i}}\overline{\mu}^{\boldsymbol{i}} \\
    =\boldsymbol{x}\mu\overline{\mu}^{\boldsymbol{i}}+A\overline{\boldsymbol{x}}^{\boldsymbol{i}}\overline{\mu}^{\boldsymbol{i}}=(\boldsymbol{x}\mu
    +A\overline{\boldsymbol{x}}^{\boldsymbol{i}})\overline{\mu}^{\boldsymbol{i}}=\boldsymbol{y}\overline{\mu}^{\boldsymbol{i}}.
    $$

    In either case (1) or (2), we have shown there exist some nonzero vector $\boldsymbol{v}\in\mathbb{H}^n$ and a quaternion $\alpha\in\mathbb{H}$ with
    $\alpha \overline{ \alpha }^{\boldsymbol{i}} =\lambda$ such that $A \overline{ \boldsymbol{v} }^{\boldsymbol{i}} =\boldsymbol{v}\alpha$. We may assume that $\boldsymbol{v}$ is a unit vector without
    loss of generality. Also, we can choose a unit modulus quaternion p such that
    $p\alpha\ (\overline{ p }^{\boldsymbol{i}})^{-1}=\sqrt{\lambda}$ by Lemma \ref{jintian}, then
    $$
    A(\overline{ \boldsymbol{v} }^{\boldsymbol{i}}(\overline{ p }^{\boldsymbol{i}})^{-1})=(\boldsymbol{v}p^{-1})\sqrt{\lambda},
    $$
it is obvious that the vector $\boldsymbol{v}p^{-1}$ is still a unit vector.

    Therefore, if $\lambda$ is a nonnegative right eigenvalue of $A \overline{ A }^{\boldsymbol{i}}$, then we can always find a unit vector $\boldsymbol{v}$ and a
    nonnegative number $\sigma=\sqrt{\lambda}$ such that 
    $$
    A \overline{ \boldsymbol{v} }^{\boldsymbol{i}} =\boldsymbol{v}\sigma=\sigma\boldsymbol{v}.
    $$
     Now we extend
    this vector $\boldsymbol{v}$ to an orthonormal basis $\{\boldsymbol{v}, \boldsymbol{v}_2, \cdots, \boldsymbol{v}_n\}$ of $\mathbb{H}^n$ by Theorem \ref{6.11}
     and let $V_{1}$ be the unitary matrix that has these vectors as columns. The first column of the matrix
     $(\overline{V_1})^TA \overline{ V_{1} }^{\boldsymbol{i}}$ has entries
     $$
     \boldsymbol{v}_l^*A \overline{ \boldsymbol{v} }^{\boldsymbol{i}} =\sigma\boldsymbol{v}_l^*\boldsymbol{v}=\sigma\delta_{l1}$$
      because of orthonormality and the
     relation $A \overline{ \boldsymbol{v} }^{\boldsymbol{i}} =\sigma\boldsymbol{v}$.

     Thus, all but the first entries in the first column of
     $(\overline{V_1})^TA {\overline{ V_1 }^{\boldsymbol{i}}}$ must be zero (the first entry might also be zero). If we write this matrix in
     partitioned form as
     $$(\overline{V_1})^TA{\overline{ V_1 }^{\boldsymbol{i}}}=
     \begin{pmatrix}
     \sigma & W^T  \\
     0 & A_2
     \end{pmatrix},$$
     where $\sigma\ge0,\ W\in\mathbb{H}^{(n-1)},\ A_2\in M_{n-1}{(\mathbb{H})}$. We can see that

     $$$$
     $$((\overline{V_1})^TA {\overline{ V_1 }^{\boldsymbol{i}}})\times {\overline{ (\overline{{V_1}})^{T}A{\overline{ V_1 }^{\boldsymbol{i}}} }^{\boldsymbol{i}}}
     =V_1^*A \overline{ A }^{\boldsymbol{i}} V_1=
     \begin{pmatrix}
     \sigma^2 & \sigma {(\overline{ W }^{\boldsymbol{i}}})^T+W^T \overline{ A_2 }^{\boldsymbol{i}} \\
     0 & A_2 \overline{ A_2 }^{\boldsymbol{i}}
     \end{pmatrix}.$$
     It can be easily verified that the right eigenvalues of $A{\overline{A}^{\boldsymbol{i}}}$ (all nonnegative by assumptions) are therefore $\sigma^2$ together with the right eigenvalues of
     $A_2\overline{ A_2 }^{\boldsymbol{i}}$.

     We conclude that $A_2\in M_{n-1}(\mathbb{H})$ obtained by this process of reduction also has the property that all the right eigenvalues of
     $A_2 \overline{ A_2 }^{\boldsymbol{i}}$ are nonnegative. The process of reduction can be repeated with $A_2$ and its successors at most $n-1$ times to obtain
     $$ (\overline{V_{n-1}})^T\cdots(\overline{V_{2}})^T(\overline{V_{1}})^TA
     {\overline{ V_1 }^{\boldsymbol{i}}}\cdots {\overline{ V_2 }^{\boldsymbol{i}}} {\overline{ V_{n-1} }^{\boldsymbol{i}}}
     =\begin{pmatrix}
     \sigma_1  & * & \cdots & * \\
 0 & \sigma_2 & \cdots & * \\
 \vdots & \vdots & \ddots & \vdots \\
 0 & 0 & \cdots & \sigma_n
     \end{pmatrix}
     =\Delta,$$
     where $\Delta$ is upper triangular with nonnegative main diagonal entries $\sigma_i,\ i=1,\cdots, n$. If we set
     $U=V_1V_2\cdots V_{n-1}$, then 
     $$
     A=U\Delta U^{*\boldsymbol{i}}.
     $$
Finally, $\Delta\overline{\Delta}^{\boldsymbol{i}}$ is upper triangular with diagonal entries $\sigma_{1}^{2},\cdots,\sigma_{n}^{2}$. This implies all of its right eigenvalues are nonnegative real numbers. Thus $(a)$ holds.
   \end{proof}

\begin{example}
Let $A=
\begin{pmatrix}
1 & 0 \\
\boldsymbol{i} & 1 \\
\end{pmatrix}$, the right eigenvalues of $A\overline{A}^{\boldsymbol{i}}$ are $\lambda_{1}=1$ and $\lambda_{2}=1$.
 Then there exists a unitary matrix
$U=
\begin{pmatrix}
0 & 1 \\
1 & 0
\end{pmatrix}$ and an upper triangular matrix $\Delta=
\begin{pmatrix}
1 & \boldsymbol{i} \\
0 & 1
\end{pmatrix}$ such that $A=U\Delta U^{*\boldsymbol{i}}$.
\end{example}

Based on Theorem \ref{3.3}, we obtain a factorization analogous to the Takagi's factorization under unitarily $\boldsymbol{i}$-equivalence for quaternion matrix. In fact, the result can be inferred from Corollary \ref{111}.

\begin{corollary}\label{UIF}(Quaternion Takagi's factorization)
Let $A\in M_n(\mathbb{H})$ be an $\boldsymbol{i}$-Hermitian matrix, then there exists a unitary matrix $U\in M_n(\mathbb{H})$ and a nonnegative diagonal matrix
$\Sigma$  such that $A=U\Sigma U^{*\boldsymbol{i}}$.
\end{corollary}

    \begin{proof}
    If $\overline{ A }^{\boldsymbol{i}} =A^{\ast}$, then $A \overline{ A }^{\boldsymbol{i}}=AA^*$. If $\boldsymbol{x}\ne\boldsymbol{0}$ is any right eigenvector of $AA^*$, that is, $AA^*\boldsymbol{x}=\boldsymbol{x}\lambda$, then
    $$\boldsymbol{x}^*(\boldsymbol{x}\lambda)=\boldsymbol{x}^*AA^*\boldsymbol{x}=({A^*\boldsymbol{x}})^*(A^*\boldsymbol{x}).$$
    Since $\boldsymbol{y}^*\boldsymbol{y}\ge0$ for $\boldsymbol{y}\in\mathbb{H}^n$ with $\boldsymbol{y}^*\boldsymbol{y}=0$ if and only if $\boldsymbol{y}=\boldsymbol{0}$, we see that
    $$\lambda=\frac{(A^*\boldsymbol{x})^*(A^*\boldsymbol{x})}{\boldsymbol{x}^*\boldsymbol{x}}\ge 0.$$
    Thus, all the right eigenvalues of $A \overline{ A }^{\boldsymbol{i}}$ are nonnegative whenever $\overline{ A }^{\boldsymbol{i}}=A^{\ast}$.
Theorem \ref{3.3} guarantees there is a unitary matrix $U\in M_n(\mathbb{H})$  and an upper triangular $\Sigma\in M_n(\mathbb{H})$ with
$$\Sigma=\begin{pmatrix}
\sigma_1  & * & \cdots & * \\
 0 & \sigma_2 & \cdots & * \\
 \vdots & \vdots & \ddots & \vdots \\
 0 & 0 & \cdots & \sigma_n
\end{pmatrix},$$
$\sigma_i\ge0,\ i=1, \cdots, n$ such that $A=U\Sigma U^{*\boldsymbol{i}}$, this is equivalent to $\overline{ A }^{\boldsymbol{i}}=\overline{ U }^{\boldsymbol{i}} \overline{ \Sigma }^{\boldsymbol{i}}U^{\ast},~~A^{\ast}= \overline{ U }^{\boldsymbol{i}} \Sigma^{\ast}U^{\ast}$, therefore, $ \overline{ \Sigma }^{\boldsymbol{i}} =\Sigma^{\ast}$, but $\Sigma$ is upper triangular, thus $\Sigma$ is diagonal. Since
$$
A^{*}A=\overline{ U }^{\boldsymbol{i}}\Sigma^{2}U^{*\boldsymbol{i}},
$$
and $\overline{ U }^{\boldsymbol{i}}$ is unitary, it follows that the diagonal entries of $\Sigma$ are the singular values
of $A$.
 \end{proof}
Note that if the matrix $A$ in Corollary \ref{UIF} is complex, then $A$ is Hermitian actually, thus by Corollary \ref{UIF} we find every complex Hermitian matrix can be unitarily $\boldsymbol{i}$-congruent to a real nonnegative diagonal matrix $\Lambda_1$ in $M_{n}(\mathbb{H})$, this is different from the truth that $A$ can be unitarily equivalent to a real diagonal matrix $\Lambda_2$ because the elements of $\Lambda_2$ may not be nonnegative.
\begin{example}
Let $A=
\begin{pmatrix}
1 & \boldsymbol{i} \\
-\boldsymbol{i} & 2 \\
\end{pmatrix}$, then $A$ is $\boldsymbol{i}$-Hermitian.
 Then there exists a unitary matrix
$U=\frac{1}{\sqrt{10+2\sqrt5}}
\begin{pmatrix}
2 & \sqrt5+1 \\
(-1-\sqrt5)\boldsymbol{i} & 2\boldsymbol{i}
\end{pmatrix}$ and a nonnegative diagonal matrix $\Delta=
\begin{pmatrix}
\frac{3+\sqrt5}{2} & 0 \\
0 & \frac{3-\sqrt5}{2}
\end{pmatrix}$ such that $A=U\Delta U^{*\boldsymbol{i}}$.
\end{example}
\begin{remark}
Our result in Corollary \ref{UIF} is actually the case in Theorem 3 of F. Zhang's paper \cite{HZ2012} when $\eta$ is equal to $\boldsymbol{i}$.

\end{remark}

\subsection{i-polar decomposition}
Inspired by the classical polar decomposition of quaternion matrices, we aim to propose a similar decomposition of quaternion matrices  through the $\boldsymbol{i}$-Hermitian matrices and unitary matrices. Hence, in this section, we show the factorization of a square quaternion matrix $A$ of the form $A=SU$, where $S\in M_{n}(\mathbb{H})$ is an $\boldsymbol{i}$-Hermitian matrix and $U\in M_{n}(\mathbb{H})$ is unitary. This factorization is named an $\boldsymbol{i}$-polar decomposition. We prove that every quaternion matrix $A$ has an $\boldsymbol{i}$-polar decomposition.

\begin{proposition}[\cite{Z}]
Let $A\in M_{n}(\mathbb{H})$, then there exists a Hermitian matrix $P$ and a unitary matrix $U$ such that $A=PU$.
\end{proposition}
The decomposition $A=PU$ is called the polar decomposition of a quaternion matrix $A$.
Let $A\in M_{n}(\mathbb{H})$, then there exist unitary matrices $X, Y$ and a diagonal matrix $\Sigma$ such that
\begin{equation}\label{tt}
A=X\Sigma Y^{*},
 \end{equation}
 the columns of $X$ are right eigenvectors of $AA^{*}$, the diagonal entries of $\Sigma$ are the square roots of the corresponding right eigenvalues and the columns of $Y$ are right eigenvectors of $A^{*}A$. Equation (\ref{tt}) is called the singular value decomposition of $A$.

The polar decomposition of $A$ can easily be obtained from its singular value decomposition. Specifically, rewrite (\ref{tt}) as
\begin{equation}\label{zaoa}
A=(X\Sigma X^{*})(XY^{*}),
\end{equation}
observe that $U=XY^{*}$ is a unitary matrix and $P=X\Sigma X^{*}$ is Hermitian.

By Corollary \ref{UIF} we know that for every $\boldsymbol{i}$-Hermitian quaternion matrix $A$, there exists a unitary matrix $U$ and a real diagonal matrix $\Sigma$ such that

\begin{equation}\label{zagai}
 \Sigma=  U^* A \overline{ U }^{\boldsymbol{i}} = \mathrm{diag} (\lambda_1, \lambda_2, \cdots, \lambda_n).
 \end{equation}
The decomposition (\ref{zagai}) becomes a singular value decomposition of $A$ when we choose $\Sigma$ nonnegative.

\begin{proposition}\label{12294}
Let $A\in M_{n}(\mathbb{H})$, then there exists an $\boldsymbol{i}$-Hermitian matrix $S$ and a unitary matrix $U$ such that $A=SU$.
\end{proposition}
\begin{proof}
Assume $A=X\Sigma Y^{*}$ is the singular decomposition of $A$, we rewrite it as
 $$
 A=(X\Sigma\overline{X^{*}}^{\boldsymbol{i}})(\overline{X}^{\boldsymbol{i}}Y^{*}),
 $$
 note that $S=X\Sigma\overline{X^{*}}^{\boldsymbol{i}}$ is an $\boldsymbol{i}$-Hermitian matrix and $U=\overline{X}^{\boldsymbol{i}}Y^{*}$ is unitary. We call $A=SU$ the $\boldsymbol{i}$-polar decomposition($\boldsymbol{i}$-HUPD) of $A$.
\end{proof}
\begin{remark}
If the matrix $A$ is complex, then the $\boldsymbol{i}$-polar decomposition of $A$ in Proposition \ref{12294} is exactly the classic polar decomposition.
\end{remark}
\begin{lemma}
Let $A\in M_{n}(\mathbb{H})$ and $A=SU$ be an arbitrary $\boldsymbol{i}$-polar decomposition of $A$. Then
\begin{equation}\label{eng?}
S\overline{S}^{\boldsymbol{i}}=AA^*.
\end{equation}
\end{lemma}
One of the applications of polar decomposition of complex matrices is that it can provide the necessary and sufficient conditions for a complex matrix to be normal \cite{GJSW87}. A similar conclusion can also be reached for quaternion matrices.
\begin{proposition}\label{yaya}
Let $A$ be a $n\times n$ matrix in $M_n(\mathbb{H})$ and $A=PU$ is the polar decomposition of $A$ where $P$ is a positive semidefinite Hermitian matrix and $U$ is a unitary matrix. Then the following are equivalent.

\begin{enumerate}
    \item $A$ is a normal matrix,
    \item $AP=PA$,
    \item $AU=UA$,
    \item $PU=UP$.
\end{enumerate}
\end{proposition}

$\boldsymbol{i}$-conjugate normal matrices play the same role in the theory of unitary $\boldsymbol{i}$-congruence as normal matrices and conjugate-normal matrices do with respect to unitary equivalence in $M_{n}(\mathbb{H})$ and unitary congruence in $M_{n}(\mathbb{C})$.
In this section, we prove the following theorem. First we introduce a definition that will be used in the proof Theorem \ref{12293}.
\begin{definition}
Let $A, B\in M_{n}(\mathbb{H})$, if $A\overline{B}^{\boldsymbol{i}}=B\overline{A}^{\boldsymbol{i}}$, then we say $A$ is $\boldsymbol{i}$-commute with $B$.
\end{definition}
We know that if two quaternion matrices $A$ and $B$ is commutative, that is, $AB=BA$, then the matrices $A_{1}=UAU^{*}$ and $B_{1}=UBU^{*}$ are still commutative for any unitary matrix $U$. Analogously, the $\boldsymbol{i}$-commutation between two quaternion matrices has the similar property.
\begin{lemma}
The $\boldsymbol{i}$-commutation is preserved by unitary $\boldsymbol{i}$-congruence.
\end{lemma}

We shall call a factorization
\[
    A=SU
\]
a left $\boldsymbol{i}$-Hermitian-unitary polar decomposition
(left $\boldsymbol{i}$-HUPD) of \(A\) if \(S\) is $\boldsymbol{i}$-Hermitian,
that is,
\[
    S=S^{*\boldsymbol{i}},
\]
and \(U\) is unitary.

The following is the main theorem of this part.
\begin{theorem}\label{12293}
Let \(A\in M_n(\mathbb H)\). Then the following assertions are
equivalent. 
\begin{enumerate}
\item [(1)] $A$ is $\boldsymbol{i}$-conjugate normal;
\item [(2)] There exists a left $\boldsymbol{i}$-HUPD 
$$
A=SU
$$
 such that
\begin{equation}\label{zh}
A\overline{S}^{\boldsymbol{i}}=S\overline{A}^{\boldsymbol{i}};
\end{equation}
\item [(3)] There exists a left $\boldsymbol{i}$-HUPD 
$$
A=SU
$$
 such that
\begin{equation}\label{zhh}
U(\overline{S}^{\boldsymbol{i}}S)=(\overline{S}^{\boldsymbol{i}}S)U;
\end{equation}
\item [(4)]There exists a left $\boldsymbol{i}$-HUPD 
$$
A=SU
$$
 such that
\begin{equation}\label{zhi}
U(A^{*}A)=(A^{*}A)U.
\end{equation}
\end{enumerate}
\end{theorem}
\begin{proof}

Necessity. Assume $A$ is $\boldsymbol{i}$-conjugate normal and let $F$ be the unitary matrix that transforms $A$ into the complex diagonal matrix $\widetilde{A}$ by unitary $\boldsymbol{i}$-congruence: $\widetilde{A}=F^{*}A\overline{F}^{\boldsymbol{i}}$. Consider the standard polar decomposition
\[
    \widetilde A=\widetilde P\widetilde U,
    \qquad
    \widetilde P
      =(\widetilde A\widetilde A^*)^{1/2}\geq0.
\]
Since \(\widetilde A\) is diagonal, both
\(\widetilde P\) and \(\widetilde U\) may be chosen diagonal.
In particular, \(\widetilde P\) is a nonnegative real
diagonal matrix.

Since $\widetilde{A}$ is normal, by Proposition \ref{yaya} we have
$\widetilde{A}\widetilde{P}=\widetilde{P}\widetilde{A}$ and $\widetilde{U}\widetilde{P}=\widetilde{P}\widetilde{U}$, which implies
$\widetilde{U}\widetilde{P}^2=\widetilde{P}^{2}\widetilde{U}.$ Now we reverse the transformation $\widetilde{A}=F^{*}A\overline{F}^{\boldsymbol{i}}$ in order to return to the original matrix $A$. Setting
\begin{equation}\label{shui}
S=F\widetilde{P}\overline{F^*}^{\boldsymbol{i}},\ \ U=\overline{F}^{\boldsymbol{i}}\widetilde{U}\overline{F^*}^{\boldsymbol{i}},
\end{equation}
we obtain an $\boldsymbol{i}$-$HUPD$ of $A$: $A=SU$. Since $\overline{S}^{\boldsymbol{i}}S=\overline{F}^{\boldsymbol{i}}\widetilde{P}^{2}\overline{F^*}^{\boldsymbol{i}}$, we deduce relation (\ref{zhh}) from $\widetilde{U}\widetilde{P}^2=\widetilde{P}^{2}\widetilde{U}$ and the second formula in (\ref{shui}). Next we observe that, by (\ref{eng?}) and the $\boldsymbol{i}$-conjugate normality of $A$, the equality
$$
\overline{S}^{\boldsymbol{i}}S=\overline{S\overline{S}^{\boldsymbol{i}}}^{\boldsymbol{i}}=\overline{AA^*}^{\boldsymbol{i}}=A^{*}A
$$
holds, which says that (\ref{zhi}) is the same relation as (\ref{zhh}).

For the complex matrix $\widetilde{A}$ and the real matrix $\widetilde{P}$ in $\widetilde{A}\widetilde{P}=\widetilde{P}\widetilde{A}$, commutativity and $\boldsymbol{i}$-commutativity are the same thing. However, $\boldsymbol{i}$-commutativity is preserved by unitary $\boldsymbol{i}$-congruence, which means that $A$ and $S$ must obey relation
(\ref{zh}).

Sufficiency. First, we suppose that condition $(2)$ is satisfied. Let $F$ be a unitary matrix that brings $S$ to the nonnegative real diagonal matrix $\Lambda$, that is, $\Lambda=F^{*}S\overline{F}^{\boldsymbol{i}}.$
Without loss of generality, we may regard $\Lambda$ as a special block diagonal matrix of the form
\begin{equation}\label{32}
\Lambda=\lambda_{1}I_{k_1}\oplus\lambda_{2}I_{k_2}\oplus\cdots\oplus\lambda_{m}I_{k_m},
\end{equation}

where
\[
 \lambda_1>\cdots >\lambda_m\geq 0
    \quad\text{and}\quad
    \lambda_r\neq\lambda_s\ \text{for }r\neq s.
\]
Set
\[
    \widetilde A=F^*A\overline{F}^{\boldsymbol{i}}.
\]
No triangularity of \(\widetilde A\) is assumed.

Under the same transformation, the relation
\[
    A\overline{S}^{\boldsymbol{i}}=S\overline{A}^{\boldsymbol{i}}
\]
becomes
\[
    \widetilde A\Lambda
    =\Lambda\overline{\widetilde A}^{\boldsymbol{i}}.
\]
Write
\[
    \widetilde A=(\widetilde A_{rs})_{r,s=1}^m
\]
according to the block decomposition of \(\Lambda\). Then
\[
    \widetilde A_{rs}\lambda_s
    =\lambda_r\overline{\widetilde A_{rs}}^{\boldsymbol{i}}.
\]
Taking Frobenius norms and using
\[
    \|\overline{\widetilde A_{rs}}^{\boldsymbol{i}}\|_F
    =\|\widetilde A_{rs}\|_F,
\]
we obtain
\[
    \lambda_s\|\widetilde A_{rs}\|_F
    =\lambda_r\|\widetilde A_{rs}\|_F.
\]
Since \(\lambda_r\neq\lambda_s\) whenever \(r\neq s\), it follows
that
\[
    \widetilde A_{rs}=0
    \qquad (r\neq s).
\]
Therefore,
\begin{equation}\label{34}
\widetilde{A}=\widetilde{A}_{11}\oplus \widetilde{A}_{22}\oplus\cdots\oplus\widetilde{A}_{mm}.
\end{equation}

Setting $\widetilde{U}=\overline{F^*}^{\boldsymbol{i}}U\overline{F}^{\boldsymbol{i}},$ we have
\begin{equation}\label{36}
\widetilde{A}=\Lambda\widetilde{U}.
\end{equation}

If $\lambda_m>0$(i.e., A is invertible), then (\ref{36}), combined with (\ref{32}) and (\ref{34}), implies that $\widetilde{U}$ is a block diagonal matrix of the same type as (\ref{34}) and
\begin{equation}\label{37}
\widetilde{A}_{ii}=\lambda_{i}\widetilde{U}_{ii},\ \ \ i=1,2,\cdots,m.
\end{equation}

Note that a scalar multiple of a unitary matrix is both a normal and an $\boldsymbol{i}$-conjugate normal matrix. Thus, being a direct sum of the $\boldsymbol{i}$-conjugate normal blocks $\widetilde{A}_{ii}$, the matrix $\widetilde{A}$ itself is $\boldsymbol{i}$-conjugate normal, hence A is $\boldsymbol{i}$-conjugate normal too.

If $\lambda_m=0$. We deduce from (\ref{36}) that
$\widetilde{U}_{ij}=0$ for $i=1,2,\cdots,m-1$ and $j\neq i$. Since $\widetilde{U}$ is unitary, this implies that $\widetilde{U}_{mj}=0$ for $j=1,2,\cdots,m-1$. The rest of the argument is the same as above, the only distinction being that $\widetilde{A}_{mm}=\lambda_{m}U_{mm}=0$. Then the matrix $\widetilde{A}$ itself is $\boldsymbol{i}$-conjugate normal, hence A is $\boldsymbol{i}$-conjugate normal too.

Now assume $(3)$ is satisfied. Now we apply to $A, S$ and $U$ transformations $\widetilde{A}=F^{*}A\overline{F}^{\boldsymbol{i}}$, $\Lambda=F^{*}S\overline{F}^{\boldsymbol{i}}$ and $\widetilde{U}=\overline{F^*}^{\boldsymbol{i}}U\overline{F}^{\boldsymbol{i}}$, respectively. Since
$$
\Lambda^2=\overline{\Lambda}^{\boldsymbol{i}}\Lambda=(\overline{F^*}^{\boldsymbol{i}}\overline{S}^{\boldsymbol{i}}S\overline{F}^{\boldsymbol{i}}),
$$
relation (\ref{zhh}) transforms into $\widetilde{U}\Lambda^2=\Lambda^{2}\widetilde{U}$.

It follows that $\widetilde{U}$ is block diagonal: $\widetilde{U}=\widetilde{U}_{11}\oplus \widetilde{U}_{22}\oplus\cdots\oplus\widetilde{U}_{mm}$. The rest of the proof is as in case $(1)$.

Finally, suppose that $(4)$ is true. Substituting $A=SU$ into (\ref{zhi}), we have
$$
\overline{S}^{\boldsymbol{i}}S=U^{*}\overline{S}^{\boldsymbol{i}}SU.
$$
Thus, (\ref{zhi}) is the same relation as (\ref{zhh}), which completes the proof.
\end{proof}

\section*{Statements and Declarations}

\noindent \textbf{Ethics approval}

\noindent Not applicable.

\noindent \textbf{Competing interests}

\noindent The authors declare that there is no conflict of interest or competing interest.

\noindent \textbf{Authors' contributions}

\noindent All authors contributed equally to this work.

\noindent \textbf{Availability of data and materials}

\noindent Data sharing is not applicable to this article as no data sets were generated or analyzed during the current study.


\begin{thebibliography}{000}
\bibitem{AI2003}
Yu. A$\text{l}^\prime$pin and Kh. Ikramov. On the unitary similarity of matrix families. Math. Notes. 74(6):772--782, 2003.

\bibitem{B}
J.~Brenner. Matrices of quaternions. Pac. J. Math. 1(3):329--335, 1951.

\bibitem{DS2008}
D.~Dokovi\'{c} and B.~Smith. Quaternionic matrices: Unitary similarity, simultaneous
  triangularization and some trace identities. Linear Algebra Appl. 428(4):890--910, 2008.

\bibitem{GJSW87}
R.~Grone, C.~Johnson, E.~Sa, and H.~Wolkowicz. Normal matrices. Linear Algebra Appl. 87:213--225, 1987.

\bibitem{HLM1974}
F.~Herbut, P.~Loncke, and M.~Vuji\v{c}i\'{c}. Canonical form for matrices under unitary congruence transformations.
  {II}. {C}ongruence-normal matrices. SIAM J. Appl. Math. 26:794--805, 1974.

\bibitem{HR}
R.~Horn and C.~Johnson. Matrix analysis. Cambridge University Press. 1985.

\bibitem{HZ2012}
R.~Horn and F.~Zhang. A generalization of the complex autonne-takagi factorization to
  quaternion matrices. Linear Multilinear A. 60:1239--1244, 2012.

\bibitem{H}
L.~Huang. Consimilarity of quaternion matrices and complex matrices. Linear Algebra Appl. 331(1-3):21--30, 2001.

\bibitem{I}
Kh. Ikramov. On complex matrices that are unitarily similar to real matrices. Math. Notes. 87(5-6):821--827, 2010.

\bibitem{JiangCL}
T.~Jiang, X.~Cheng, and S.~Ling. An algebraic relation between consimilarity and similarity of
  quaternion matrices and applications. J. Appl. Math. 795203, 5 pp, 2014.

\bibitem{RE1944}
R.~Johnson. On the equation $\chi \alpha = \gamma \chi + \beta$ over an algebraic division ring. B. Am. Math. Soc. 50:202--207, 1944.

\bibitem{WS1940}
W.~Specht. Zur theorie der matrizen. {I}{I}. Jahresber. Deutsch. Math-Verein. 50:19--23, 1940.

\bibitem{VHV1972}
M.~Vuji\v{c}i\'{c}, F.~Herbut, and G.~Vuji\v{c}i\'{c}. Canonical form for matrices under unitary congruence transformations.
  {I}. {C}onjugate-normal matrices. SIAM J. Appl. Math. 23:225--238, 1972.

\bibitem{W}
L.~Wolf. Similarity of matrices in which the elements are real quaternions. B. Am. Math. Soc. 42:737--743, 1936.

\bibitem{Z}
F.~Zhang. Quaternions and matrices of quaternions. Linear Algebra Appl. 251:21--57, 1997.

\bibitem{Z07}
F.~Zhang.  Ger$\check{s}$gorin type theorems for quaternionic matrices. Linear Algebra Appl. 424(1):139--153, 2007.
\end{thebibliography}
\end{document}